\documentclass[10pt,a4paper]{article}

\usepackage[utf8]{inputenc}
\usepackage[english]{babel}
\usepackage[T1]{fontenc}
\usepackage{amsmath}
\usepackage{amsfonts}
\usepackage{amssymb}
\usepackage{amsthm}
\usepackage{stmaryrd}
\usepackage{tikz}
\usepackage{tikz-cd}
\usepackage{enumitem}
\usepackage{todonotes} 
\usepackage{graphicx}
\usepackage{hyperref}

\DeclareMathOperator{\Hom}{Hom}
\DeclareMathOperator{\End}{End}
\DeclareMathOperator{\Aut}{Aut}
\DeclareMathOperator{\Ker}{Ker}
\DeclareMathOperator{\Cok}{Coker}
\DeclareMathOperator{\Ima}{Im}

\DeclareMathOperator{\inte}{int}

\DeclareMathOperator{\Trd}{Trd}

\DeclareMathOperator{\Br}{Br}

\DeclareMathOperator{\Id}{Id}

\DeclareMathOperator{\Sym}{Sym}
\DeclareMathOperator{\Sim}{Sim}
\DeclareMathOperator{\Iso}{Iso}
\DeclareMathOperator{\rdim}{rdim}

\DeclareMathOperator{\LaxHom}{LaxHom}

\newcommand{\isom}{\stackrel{\sim}{\rightarrow}}
\newcommand{\Isom}{\stackrel{\sim}{\longrightarrow}}

\newcommand{\ie}{\textit{i.e.}}
\newcommand{\Z}{\mathbb{Z}}
\newcommand{\N}{\mathbb{N}}

\newcommand{\pfis}[1]{\langle\!\langle #1\rangle\!\rangle}
\newcommand{\To}{\longrightarrow}

\newcommand{\fdiag}[1]{\langle #1\rangle}
\newcommand{\ens}[2]{\{ #1\,|\, #2\}}

\newcommand{\tld}{\widetilde}
\newcommand{\eps}{\varepsilon}

\newcommand{\bitem}{\item[$\bullet$]}
\newcommand{\pgq}{\geqslant}

\newcommand{\CBrh}[1][K]{\mathbf{Br}_h(#1)}
\newcommand{\CBrhu}{\mathbf{Br}_h(K,\iota)}
\newcommand{\Zd}{\Z/2\Z}

\newcommand{\mcat}[1][R]{#1\text{-}\mathbf{Mod}}
\newcommand{\acat}[1][R]{#1\text{-}\mathbf{Alg}}
\newcommand{\cacat}[1][R]{#1\text{-}\mathbf{CommAlg}}
\newcommand{\mor}[1]{\stackrel{#1}{\rightsquigarrow}}

\renewcommand{\phi}{\varphi}
\renewcommand{\bar}{\overline}
\renewcommand{\hat}{\widehat}

\newcommand{\cat}[1]{\langle #1\rangle}

\newcommand{\foncdef}[5]{\begin{array}{rrcl}
#1 : & #2 & \To & #3 \\
 & #4 & \longmapsto & #5
\end{array}}

\newcommand{\anonfoncdef}[4]{\begin{array}{rcl}
#1 & \To & #2 \\
#3 & \longmapsto & #4
\end{array}}

\newtheorem{thm}{Theorem}[section]
\newtheorem{prop}[thm]{Proposition}
\newtheorem{coro}[thm]{Corollary}
\newtheorem{lem}[thm]{Lemma}
\newtheorem{defi}[thm]{Definition}

\theoremstyle{definition}
\newtheorem{rem}[thm]{Remark}
\newtheorem{ex}[thm]{Example}

\author{Nicolas Garrel}
\title{Mixed Witt rings of algebras with involution}
\date{}

\begin{document}

\maketitle

\textbf{Abstract:} Although there is no natural internal product for hermitian
forms over an algebra with involution of the first kind, we describe how to multiply two
$\eps$-hermitian forms to obtain a quadratic form over the base field. This
allows to define a commutative graded ring structure by taking together bilinear
forms and $\eps$-hermitian forms, which we call the mixed Witt ring of an algebra
with involution. We also describe a less powerful version of this construction
for unitary involutions, which still defines a ring, but with a grading
over $\mathbb{Z}$ instead of the Klein group.

We first describe a general framework for defining graded rings out of monoidal
functors from monoidal categories with strong symmetry properties to categories
of modules. We then give a description of such a strongly symmetric category $\CBrhu$
which encodes the usual hermitian Morita theory of algebras with involutions over a field
$K$.

We can therefore apply the general framework to $\CBrhu$ and the Witt group functors
to define our mixed Witt rings, and derive their basic properties, including explicit formulas for
products of diagonal forms in terms of involution trace forms, explicit computations for the
case of quaternion algebras, and reciprocity formulas relative to scalar extensions.

We intend to describe in future articles further properties of those rings,
such as a $\lambda$-ring structure, and relations with the Milnor conjecture
and the theory of signatures of hermitian forms.

MSC classes: 11E81 (primary), 19G12, 16W10 (secondary).

\pagebreak

\section*{Introduction}

\subsection*{From quadratic forms to hermitian Morita theory}

In the 19th century, quadratic forms were the object of many investigations,
notably by algebraists such as Gauss, Minkowski or Kronecker, but were mostly
given an arithmetic flavour. The birth of the algebraic theory of quadratic forms
over arbitrary fields (of characteristic not $2$) stems from the seminal 1936 paper 
of Ernst Witt (\cite{Wit}), although a completely unnoticed 1907 paper of Dickson 
(\cite{Dick}) had actually already made significant advances (see Scharlau's comment 
in \cite{Sch2}). A key insight of Witt was to shift from the study of individual 
quadratic forms to the study of the structure they form as a whole. He considers the 
set $W(K)$ of isometry classes of quadratic forms over $K$ up to what is now called
\emph{Witt equivalence}, and shows that the direct sum and tensor product of 
quadratic spaces endow this set with a natural commutative ring structure.

Though quadratic forms over fields were still objects of study, notably through
the study of various arithmetic invariants, this algebraic theory of Witt truly
started expanding when it was rediscovered in the mid 60s by Pfister (see \cite{Pfi}),
and others such as Arason. It then became abundantly clear that the algebraic
structure of the Witt ring was a central element of the theory, and was connected
to many subjects of interest, such as Galois cohomology through the Milnor conjecture,
or the theory of field orderings, in relation with the much older Artin-Schreier theory.
For detailed references on the algebraic theory of quadratic forms, we can cite
\cite{Lam}, \cite{Sch} or \cite{EKM}.

Around the same time, algebraic K-theory was being developed, and the connection with
quadratic forms became apparent, as can be seen notably in the foundational
work of Bass (\cite{Bass}). This led to working over general commutative rings,
but also to consider not just the Witt ring $W(K)$ but what we will call the
Grothendieck-Witt ring $GW(K)$ (the terminology "Witt-Grothendieck ring" is equally
common); one crucial reason Witt introduced the Witt equivalence is that it 
creates additive inverses in $W(K)$, but K-theory taught us that it is very natural
to rather add formal inverses. Shortly after that, Bak (\cite{Bak}) and 
Fröhlich-McEvett (\cite{FME}) independently started to generalize bilinear forms over 
commutative rings to $\eps$-hermitian forms over (non-commutative) rings with involution
(where $\eps$ is a parameter, which is simply equal to $\pm 1$ in the classical
context of fields with no involution, corresponding to the theory of symmetric
and anti-symmetric bilinear forms); a more advanced reference is found in \cite{Knu}.
In particular, Fröhlich and McEvett define
Witt and Grothendieck-Witt \emph{groups} $W^\eps(A,\sigma)$ and $GW^\eps(A,\sigma)$;
furthermore, they develop an analogue for hermitian modules of the classical Morita 
theory over rings, which they aptly call hermitian Morita theory. This theory is the
central tool used in this article, and we review it in Section \ref{sec_brau}, although
in a different form. 

Prominent among rings with involutions are the (finite-dimensional) central simple 
algebras with involution over fields, which we will thereafter simply call
"algebras with involution". Their study was initiated by Albert in \cite{Alb}
with the context of Riemann surfaces in mind, but found a renewed interest when
Weil (\cite{Wei}) showed that they can be used to describe the simple algebraic
groups over arbitrary fields. A comprehensive treatment of those algebras and
their hermitians forms is given in the reference monograph \cite{BOI} by
Knus-Merkurjev-Rost-Tignol.

\subsection*{The idea of the mixed Witt ring}

One important caveat in all those generalizations is that when we work over
non-commutative rings we can only speak about (Grothendieck-)Witt \emph{groups} 
and not \emph{rings}. This is simply because there is no tensor product of modules 
over non-commutative rings. Rather, the tensor product of two $A$-modules (for example
on the right)
is an $(A\otimes A)$-module. This being said, there is still some compatibility
between Witt groups and tensor products: if $(A,\sigma)$ and $(B,\tau)$ are
algebras with involution over some field with involution $(K,\iota)$, the
tensor product induces maps
\[ GW^\eps(A,\sigma)\times GW^{\eps'}(B,\tau)\To 
GW^{\eps\eps'}(A\otimes_K B,\sigma\otimes \tau) \]
(here $\eps,\eps'\in K^\times$ are parameters such that $\eps\iota(\eps)=1$). This
is not surprising if one interprets the Grothendieck-Witt group as the degree
$0$ part of hermitian K-theory (see \cite{Bak2} for instance), and it expresses
the fact that (Grothendieck-)Witt groups are \emph{monoidal functors} on some
category (which will be made precise).

In the special case of algebras with involutions of the first kind (so $\iota=\Id$
in the notation above), a remarkable phenomonenon is that not only do the algebras
"have order $2$" with respect to Brauer-equivalence, but we have an explicit canonical
hermitian Morita equivalence between $(A\otimes_K A,\sigma\otimes \sigma)$ and
$(K,\Id)$, which means that by general considerations of Morita theory we get
a canonical isomorphism 
\[  GW(A\otimes_K A,\sigma\otimes\sigma) \Isom GW(K). \]
If we combine this with the natural product
\[ GW^\eps(A,\sigma)\times GW^\eps(A,\sigma) \to GW(A\otimes_K A,\sigma\otimes \sigma) \]
we get a canonical "mixed" product
\[ GW^\eps(A,\sigma)\times GW^\eps(A,\sigma) \to GW(K). \]

In other words, we \emph{can} define the product of two $\eps$-hermitian forms,
but the result is not a hermitian form but rather a quadratic form over the base field.
It is in that sense that we mean the product is "mixed". Using the basic fact that $GW(K)$ 
is a ring and $GW^\eps(A,\sigma)$ is a $GW(K)$-module, this allows us to construct a 
$\Z/2\Z$-graded commutative ring
\[ \tld{GW}^\eps(A,\sigma) = GW(K) \oplus GW^\eps(A,\sigma). \]
Actually, it turns out that for functoriality reasons it is much more convenient
to bundle together hermitian and anti-hermitian forms, and rather consider
the \emph{mixed Grothendieck-Witt ring}
\[ \tld{GW}(A,\sigma) = GW(K) \oplus GW^{-1}(K) \oplus GW^1(A,\sigma) 
    \oplus GW^{-1}(A,\sigma), \]
which is graded over $\Gamma = \Zd\times \mu_2(K)$. Of course, the same construction
holds for Witt groups, giving the \emph{mixed Witt ring}
\[ \tld{W}(A,\sigma) = W(K) \oplus W^1(A,\sigma) \oplus W^{-1}(A,\sigma). \]
A large part of this article is dedicated to showing that these are indeed
well-defined commutative rings, and that they are functorial in the sense that
any hermitian Morita equivalence between algebras with involution induces a graded
ring isomorphism.

It should be mentioned that a very similar construction was made by Lewis in
\cite{Lew}, for the special case where $A$ is a division quaternion algebra,
although many key properties, such as associativity and commutativity, are stated 
without proof. Instead of the involution trace form, Lewis uses the norm form, which is a
special feature of quaternion algebras.

We also present a less powerful construction which has the merit of working even
for involutions of the second kind: in that case $(A,\sigma)$ is not its own inverse
(up to hermitian Morita equivalence); rather its inverse is the twisted algebra
$(A^\iota,\sigma)$, where the action of $K$ on $A$ is twisted by $\iota$. Then we
can define a $\Z$-graded ring $\hat{GW}(A,\sigma)$ where the component of degree
$n$ is $GW(A^{\otimes n},\sigma^{\otimes n})$ if $n\pgq 0$ and 
$GW((A^\iota)^{\otimes (-n)},\sigma^{\otimes (-n)})$ if $n< 0$; we can "cancel out"
the $A$ and $A^\iota$ as many times as necessary, so that the product of
$x\in GW(A^{\otimes n},\sigma^{\otimes n})$ and 
$y\in GW((A^\iota)^{\otimes m},\sigma^{\otimes m})$ is in 
$GW(A^{\otimes n-m},\sigma^{\otimes n-m})$ if $n\pgq m$, and in
$GW((A^\iota)^{\otimes m-n},\sigma^{\otimes m-n})$ otherwise.

\subsection*{Content of the article}

Our approach in defining and studying our mixed rings is to highlight as much as
possible the formal aspects of the construction: it is a consequence of general
properties of symmetric monoidal categories, applied to a certain category which
encodes hermitian Morita theory. In consequence, the first part is dedicated to
general considerations on monoidal categories. 

Sections \ref{sec_mono} and  \ref{sec_sym} are a basic
introduction to symmetric monoidal categories, as can be found in any reference
on categories, such as the classic \cite{Mac}. They are here for convenience of readers
who are unfamiliar with that theory, as well as to fix some notations. 

Section \ref{sec_tors} contains the technical heart of the article: we explain how a special 
property of a symmetric monoidal category, which we call \emph{strong symmetry} 
(Definition \ref{def_strong}), allows to coherently choose inverses of objects 
(Theorem \ref{thm_inverses}), 
and to coherently handle $n$-torsion of objects (Proposition \ref{prop_torsion}). The strong symmetry 
property was introduced in \cite{Ulb}, where it was already used to get a similar result 
about inverses; our
result is basically a reformulation of \cite[4.6,4.7]{Lap}. This being said, the presentation we
give in terms of monoidal functors $\cat{\Z}\to \mathbf{C}$ is our own, and we do
believe that it sheds an interesting light on those results, by constructing a
universal category $\mathbf{C}^\times$ of inverses in $\mathbf{C}$, and defining
the relevant structure as a functor $\mathbf{C}\to \mathbf{C}^\times$ which
"coherently chooses inverses". There are many variations in the literature around
such "enhanced group structures" on monoidal categories (see \cite{Bae} for more
references), and so we give yet another one. On the other hand, Proposition \ref{prop_torsion},
which adapts the previous idea to handle $n$-torsion with monoidal functors
$\cat{\Z/n\Z}\to \mathbf{C}$ is our own, though the basic idea of the proof is
very reminiscent of \cite{JS}.

In Section \ref{sec_grad}, we explain how to construct graded rings (which is our end goal)
from monoidal functors. In fact, we explain that an $M$-graded ring is essentially
the same as a lax monoidal functor from the discrete category $\cat{M}$ to the
category of abelian groups. It is a very simple idea, but we have not found previous
uses in the literature. We give slightly
refined versions of that statement which will be adapted to the construction of our
mixed rings, in relation with the structures described in the previous section
(Corollary \ref{cor_graded_functors}).
\\

The second part is dedicated to the presentation of our version of hermitian Morita
equivalence (in the framework of algebras with involution), which has a strong
categorical flavour, so that we can combine it later with the result of the
first part.

Section \ref{sec_herm} exposes the basic theory of hermitian forms over algebras with involution,
mostly to fix the terminology and notations. It does not contain anything original
beyond some notation, and all the material can be found in \cite{BOI} and \cite{Knu}.
Particular attention should be given to Example \ref{ex_inv_trace}, which presents the canonical
Morita equivalence alluded to earlier and which is the basis of our ring construction;
it was already introduced in \cite{FME}.

In Section \ref{sec_cbrh}, we explain how to package hermitian Morita theory in a category
$\CBrhu$. The only thing new is the presentation, as the actual mathematical results
can all be found in \cite{Knu}, and actually already in \cite{FME}. The notion that
(classical) Morita theory could be elegantly expressed in a certain $2$-category whose
morphisms are bimodules is now rather classical in the field of "higher algebra"
(see for instance \cite{Dus}) and has been generalized in all sorts of direction
(for example \cite{Vit}), but to our knowledge this is the first time the hermitian
version is explicitly written down in practical terms. Note that indeed the natural
framework would be to consider a $2$-category, where the $2$-morphisms are bimodule
morphisms, but it is not needed for what follows, so to avoid unnecessary complications
we "truncate" the natural $2$-category to a plain category, at the cost of having
morphisms be \emph{isometry classes} of hermitian bimodules. Since the non-hermitian
version is sometimes called the \emph{Brauer $2$-group} of a commutative ring, we
named this category the \emph{hermitian Brauer $2$-group}.

In Section \ref{sec_autom}, we investigate the connection between the usual automorphisms of
an algebra with involution and its hermitian Morita self-equivalences. In particular,
Proposition \ref{prop_autom} shows that an algebraic automorphism induces the trivial Morita
automorphism if and only if it is inner.

This is used in Section \ref{sec_gold} to show that $\CBrhu$ is strongly symmetric (Corollary
\ref{cor_strong_sym}); the key
point is the existence of the ubiquitous \emph{Goldman element} in a central simple
algebra, since it implies that the switching automorphism of $A\otimes_K A$ is actually
inner. This is the crucial result upon which our construction relies, and although
ultimately quite simple, it seems to have gone unnoticed until now.

We show in Section \ref{sec_tors_cbrh} that for involutions of the first kind, the canonical 
equivalence defined by the involution trace form does define a coherent $2$-torsion
structure on $\CBrh$ (Theorem \ref{thm_torsion}).
\\

The third part of the article applies the general results of the first part to the
category $\CBrhu$ to define our mixed rings.

Section \ref{sec_group} simply defines the underlying (Grothendieck-)Witt groups, and checks
that they are functorial with respect to $\CBrhu$. In order to make them truly
functorial as graded objects, a harmless change of labelling is needed (the issue
being that anti-hermitian equivalences reverse the signs of hermitian forms instead
of preserving them).

In Section \ref{sec_ring}, we establish that those functors are actually monoidal, and thus
that the whole machinery of Section \ref{sec_grad} can be used. This leads to the definition
of the rings $\tld{GW}(A,\sigma)$ and $\tld{W}(A,\sigma)$ (as well as the less
interesting $\hat{GW}(A,\sigma)$ and $\hat{W}(A,\sigma)$ for unitary involutions).

Section \ref{sec_dim} presents how to handle the (reduced) dimension of hermitian modules,
in a way compatible with our graded ring structure.

In order to perform explicit computations in our rings, we describe in Section \ref{sec_diag}
how to multiply diagonal forms; it turns out the result is given by the twisted
involution trace forms introduced in \cite[Â§11]{BOI} (Proposition \ref{prop_prod_diag}). Whether 
or not this actually makes the product explicit is a matter of opinion, as those forms 
are actually quite difficult to compute in practice.

There is at least a case where we can reasonably say that the products are fully
computable, namely the case of a quaternion algebra, endowed with its canonical
involution. We work this out in Section \ref{sec_quater} (Proposition \ref{prop_prod_quater}).

Finally, the last Section \ref{sec_recip}, which is the most technical one of the third
part, is dedicated to scalar extensions, and especially to establishing a Frobenius
reciprocity formula (compare Theorem \ref{thm_recipr} with \cite[2.5.6]{Sch}).

\subsection*{Perspectives}

It should be noted that while in this article we work with algebras over fields,
basically everything carries over to Azumaya algebras with involution over
general commutative rings. The key ingredient of Corollary \ref{cor_strong_sym} will still
hold in that general context thanks to the existence of the Goldman element.
It can even be envisioned to work with Azumaya algebras over schemes, or even
locally ringed topos, as studied in \cite{FW}.

In addition to extending the algebras, we could also extend Grothendieck-Witt
rings to the whole hermitian K-theory, and actually to simple algebraic K-theory
if we forget about the involutions. The fact that such extensions could be easily
handled with very little extra legwork is a compelling argument for our general abstract 
presentation of the mechanisms at play as properties of monoidal categories.

Even considering simply algebras with involutions over fields, we intend to 
build on the present article in future work to study, among other things: a
$\lambda$-ring structure on $\tld{GW}(A,\sigma)$ (the existence of which is a strong
advantage over $\tld{W}(A,\sigma)$); an extension of Artin-Schreier theory for
the spectrum of $\tld{W}(A,\sigma)$ (related to the work of Astier and Unger in 
\cite{AU2}); applications of our structure to the construction of cohomological
invariants of algebras with involution and algebraic groups.

\subsection*{Aknowledgements}

This work was partially supported by a grant from the Simons Foundation, during
a stay at the Isaac Newton Institute. This work was also supported by the Fonds Wetenschappelijk
Onderzoek – Vlaanderen (FWO) in the FWO Odysseus Programme (project ‘Explicit Methods in
Quadratic Form Theory’, G0E6114N), the Bijzonder Onderzoeksfonds (BOF), University of Antwerp
(project BOF-DOCPRO-4, 2865).

\subsection*{Preliminaries and conventions}

We fix a base field $K$ of characteristic not 2, and we identify symmetric
bilinear forms and quadratic forms over $K$, through $b\mapsto q_b$ with
$q_b(x)=b(x,x)$. Diagonal quadratic forms are denoted $\fdiag{a_1,\dots,a_n}$,
with $a_i\in K^*$, and $\pfis{a_1,\dots,a_n}$ is the $n$-fold Pfister form
$\fdiag{1,-a_1}\otimes \cdots \otimes \fdiag{1,-a_n}$. We always assume that 
bilinear forms are nondegenerate.

All rings are associative and with unit, and ring morphisms preserve the unit.
The group of invertible elements of a ring $A$ is denoted $A^\times$.
The opposite ring of $A$ (that is, the ring with the reversed product)
is denoted $A^{op}$.
Unless otherwise specified, modules are by default modules on the \emph{right}, and
are assumed to be non-zero. Every $K$-algebra and every module over such
an algebra are required to have finite dimension over $K$. 
If $A$ and $B$ are $K$-algebras, a $B$-$A$-bimodule is always supposed to be over
$K$, meaning that the right and left actions of $K$ on $V$ coincide.
If $A$ is a central simple algebra over $K$, we write $\Trd_A:A\to K$ for the reduced
trace of $A$.

We fix an automorphism $\iota$ of $K$ such that $\iota^2=\Id$, and we 
let $k=K^\iota$ be the subfield of fixed points. If $\iota=\Id$ then 
$K=k$, but if $\iota\neq \Id$ then $K/k$ is a quadratic extension.
Note that we do not include the split case $K=k\times k$, to avoid having
to discuss it separately. The reader is encouraged to check that everything
would still work in that context.

When we say that $(A,\sigma)$ is an algebra with involution over $(K,\iota)$, we mean
that $A$ is a central simple algebra over $K$, and that $\sigma$ is an involution
on $A$ with $\sigma_{|K}=\iota$, so $\sigma$ is a $k$-algebra anti-automorphism
of $A$, with $\sigma^2=\Id_A$. 

Recall that the involution $\sigma$ is of the first kind
when $\iota= \Id$, and is of the second kind (or unitary) otherwise. An involution
of the first kind is orthogonal if $\dim_K(\Sym(A,\sigma)) = n(n+1)/2$
where $n$ is the degree of $A$, and it is symplectic if $\dim_K(\Sym(A,\sigma))=n(n-1)/2$. 
In particular, $(K,\Id)$ is an algebra with orthogonal involution. A quaternion algebra 
admits a unique symplectic involution, called its canonical involution.

We set $U(K,\iota) = 
\{\eps\in K^\times|\eps\iota(\eps)=1\}$; it is a subgroup of $K^\times$,
and when $\iota=\Id$ it is simply $\mu_2(K)$. If $\eps\in U(K,\iota)$, we define 
$\Sym^\eps(A,\sigma)$ as the set of $\eps$-symmetric elements of $A$, which satisfy 
$\sigma(a)=\eps a$. We also write $\Sym^\eps(A^\times,\sigma)$ for the set of invertible 
$\eps$-symmetric elements. 

If $L/K$ is any field extension, and $X$ is an object (algebra, module, involution,
hermitian form, etc.) over $K$, then $X_L$ is the corresponding object over $L$,
obtained by base change.

A semi-group is a set endowed with an associative binary product (so the difference
to a monoid is that the semi-group might lack an identity element). The Grothendieck 
group $G(S)$ of $S$ is the universal solution to the problem of finding a morphism $S\to G$ 
where $G$ is a group.



\section{Monoidal categories and graded rings}\label{sec_cat}

This section serves both as a quick primer (or reminder) for readers who
might be unfamiliar with monoidal categories, and as a presentation of the specific
methods used in this article to produce graded rings from those categories. We do assume
familiarity with category theory in general. We state most of the general theory without
proof, and refer to \cite{Mac} for more details.

\subsection{Monoidal categories}\label{sec_mono}

A monoidal category $(\mathbf{C},\otimes,I)$ (more precisely 
$(\mathbf{C},\otimes,I,\alpha,\lambda,\rho)$) consists of the data of: 
\begin{itemize}
  \item a category $C$;
  \item a bifunctor $\otimes: \mathbf{C}\times \mathbf{C}\to \mathbf{C}$ 
  (\ie the expression $x\otimes y$ is functorial in each variable);
  \item a distinguished \emph{unit object} $I\in \mathbf{C}$;
  \item for each triple $x,y,z\in \mathbf{C}$, a natural isomorphism 
  $\alpha_{x,y,z}: x\otimes (y\otimes z) \Isom (x\otimes y)\otimes z$ 
  called \emph{associator};
  \item for each object $x\in \mathbf{C}$, a natural isomorphism 
  $\lambda_x: I\otimes x\Isom x$ called \emph{left unitor};
  \item for each object $x\in \mathbf{C}$, a natural isomorphism
  $\rho_x: x\otimes I\Isom x$ called \emph{right unitor}.
\end{itemize}

The fact that the product is bifunctorial means that for any morphisms $f:x\to y$
and $g:x'\to y'$ we get a well-defined $f\otimes g: x\otimes x'\to y\otimes y'$.

Those various isomorphisms are required to satisfy certain coherence laws relating
one another: the triangle diagram
\[ \begin{tikzcd}
  x\otimes (I\otimes y) \drar[swap]{\Id_x\otimes \lambda_y} \arrow[rr, "\alpha_{x,I,y}"]
  & & (x\otimes I)\otimes y \dlar{\rho_x\otimes \Id_y} \\
   & x\otimes y &
\end{tikzcd} \]
and the pentagon diagram
\[ \begin{tikzcd}[column sep=large]
  x\otimes(y\otimes(z\otimes t)) \rar{\alpha_{x,y,z\otimes t}} \dar{\Id_x\otimes \alpha_{y,z,t}} &
  (x\otimes y)\otimes (z\otimes t) \rar{\alpha_{x\otimes y,z,t}} &
  ((x\otimes y)\otimes z) \otimes t \\
  x\otimes ((y\otimes z)\otimes t) \arrow[rr,"\alpha_{x,y\otimes z,t}"] & &
  (x\otimes (y\otimes z))\otimes t. \uar{\alpha_{x,y,z}\otimes \Id_t}
\end{tikzcd} \]

The associators and unitors are very much part of the data, although in most
natural examples they are the "obvious" isomorphisms, and are rarely actually
spelled out in practice. Furthermore, MacLane's coherence theorem (\cite[VII.2]{Mac}) guarantees
that under the above axioms there is no ambiguity arising if we omit all parentheses
and simplify all products with $I$, as any two identifications 
between two different expressions using associators and unitors will actually be equal.
Therefore, we will use a more relaxed style of notation, and leave all associators/unitors
completely in the background. In particular, we allow ourself to write $x^{\otimes n}$ for
any object $x$ and any $n\in \N$ (with the convention $x^{\otimes 0} = I$). We also often simply
say that "$\mathbf{C}$ is a monoidal category" when the notations are clear.

\begin{ex}
  If $R$ is a commutative ring, then $(\mcat,\otimes_R,R)$ and 
  $(\acat,\otimes_R,R)$ are monoidal categories.
\end{ex}

\begin{ex}
  If $M$ is a monoid, then the discrete category $\cat{M}$ with $M$
  as its underlying set is canonically a monoidal category, the tensor product of objects
  corresponding to the product in $M$.
\end{ex}

If $(\mathbf{C},\otimes, I)$ and $(\mathbf{D},\otimes, J)$ are monoidal categories,
a \emph{(lax) monoidal functor} from $\mathbf{C}$ to $\mathbf{D}$ consists of the data of:
\begin{itemize}
  \item a functor $F: \mathbf{C}\to \mathbf{D}$;
  \item a morphism $J\to F(I)$;
  \item for each $x,y\in \mathbf{C}$, a morphism $F(x)\otimes F(y)\to F(x\otimes y)$.
\end{itemize}

Those morphisms are once again required to satisfy some coherence laws, which we will
not spell out here, in the same spirit as the ones above. When those morphisms are 
isomorphisms, the monoidal functor is called \emph{strong} (if they are equalities, it is called 
\emph{strict}). Again, the structural morphisms in the definition are often "the obvious ones"
and are often left unnamed in practice.

\begin{ex}
  If $R\to S$ is a commutative ring morphism, then the scalar extension functors
  $\mcat\to \mcat[S]$ and $\acat\to \acat[S]$ have an obvious strong monoidal structure.
\end{ex}

\begin{ex}
  Any monoid morphism $M\to N$ induces in an obvious way a strict monoidal functor
  $\cat{M}\to \cat{N}$.
\end{ex}

A natural transformation $\phi: F\to G$ between two monoidal functors $\mathbf{C}\to \mathbf{D}$
is called \emph{monoidal} if it satisfies the commutative diagrams
\[ \begin{tikzcd}
  & J \dlar{} \drar{} & \\
  F(I) \arrow[rr, "\phi_I"] & & G(I)
\end{tikzcd} \]
and
\[ \begin{tikzcd}
  F(x)\otimes F(y) \dar \rar{\phi_x\otimes \phi_y} & G(x)\otimes G(y) \dar \\
F(x\otimes y) \rar{\phi_{x\otimes y}} & G(x\otimes y).
\end{tikzcd}  \]

Note that while being a monoidal category and being a monoidal functor are \emph{structures} put
on top of categories/functors, being a monoidal natural transformation is a \emph{property}.

Monoidal functors and natural transformations can be composed in the obvious way (the behaviour
of the structural morphisms in those compositions is straightforward). In particular, we
get the expected notions of monoidally isomorphic monoidal functors, and of monoidally equivalent
categories (in fact, monoidal categories form a $2$-category, so all the usual notions 
can apply without change). The composition of two strong (resp. strict) monoidal functors
is again strong (resp. strict).

\begin{rem}
  It is well-known (see for instance \cite[1.5.3]{EGNO}) that a strong monoidal functor is
  a monoidal equivalence if and only if it is an equivalence (in other words, if it has
  a quasi-inverse, then it also has a monoidal quasi-inverse).
\end{rem}

Given two (small) monoidal categories $\mathbf{C}$ and $\mathbf{D}$, we define 
$\Hom_{\otimes}(\mathbf{C},\mathbf{D})$ (resp. $\LaxHom_{\otimes}(\mathbf{C},\mathbf{D})$) 
as the category of strong (resp. lax) monoidal functors between the two, with monoidal 
natural transformations as morphisms. Note that any
$F\in \Hom_{\otimes}(\mathbf{C},\mathbf{D})$ is canonically isomorphic to some $F'$
such that the structural isomorphism $J\to F'(I)$ is actually the identity, and
the structural isomorphisms $F'(I)\otimes F'(x)\to F(I\otimes x)\isom F(x)$ and
$F'(x)\otimes F'(I)\to F(x\otimes I)\isom F(x)$ are then identified with the unitors
in $\mathbf{D}$. Therefore we will usually implicitly restrict to those functors, which
brings no noticeable change to the theory (they form a full subcategory of 
$\Hom_{\otimes}(\mathbf{C},\mathbf{D})$ such that the inclusion functor is an equivalence,
with a canonical quasi-inverse).

\begin{ex}
  If $M$ and $N$ are monoids, then $\Hom_{\otimes}(\cat{M},\cat{N})$
  is identified with the set of monoid morphisms $M\to N$.
\end{ex}

\subsection{Symmetric monoidal categories}\label{sec_sym}

If $(\mathbf{C},\otimes,I)$ if a monoidal category, a \emph{symmetric structure} on
$\mathbf{C}$ is the data of an isomorphism natural in $x$ and $y$
\begin{equation}
  s_{x,y}: x\otimes y \Isom y\otimes x
\end{equation}
for all $x,y\in \mathbf{C}$, which we call the switching morphism (also called ``symmetry
isomorphism''), satisfying some coherence 
axioms so that it is compatible with associators and unitors, and very importantly the 
involution axiom stating that
\begin{equation}
  s_{y,x}\circ s_{x,y} = \Id_{x\otimes y},
\end{equation}
in other words the composition
\[ x\otimes y\xrightarrow{s_{x,y}} y\otimes x \xrightarrow{s_{y,x}} x\otimes y \]
is the identity of $x\otimes y$ (without this last axiom, the category is only called
\emph{braided}).

In a symmetric monoidal category, for any objects $x_1,\dots,x_n\in \mathbf{C}$ and
any permutation $g\in \mathfrak{S}_n$, there is a well-defined induced isomorphism
\begin{equation}
  g_*: x_1\otimes\dots \otimes x_n \Isom x_{g^{-1}(1)}\otimes\dots\otimes x_{g^{-1}(n)},
\end{equation}
which can be described by applying a switching morphism for each transposition, and in particular 
there is a canonical group morphism
\begin{equation}
  \mathfrak{S}_n\to \Aut_{\mathbf{C}}(x^{\otimes n})
\end{equation}
for each object $x\in \mathbf{C}$ (in a braided category, we get a morphism from the braid
group instead, hence the name).

\begin{ex}
  Categories of modules or algebras are canonically symmetric, the switching morphisms 
  being the obvious ones.
\end{ex}

\begin{ex}
  If $M$ is a monoid, the monoidal category $\langle M\rangle$ has a symmetric structure,
  necessarily unique, if and only if $M$ is commutative.
\end{ex}

A monoidal functor $F$ between two symmetric monoidal categories $\mathbf{C}$ and $\mathbf{D}$ 
is called \emph{symmetric} (this time it is a property and not a structure) if the following 
diagram commutes for all $x,y\in \mathbf{C}$:
\begin{equation}\label{diag_sym_funct}
\begin{tikzcd} 
  F(x)\otimes F(y) \rar{s_{F(x),F(y)}} \dar & F(y)\otimes F(x) \dar \\
  F(x\otimes y) \rar{F(s_{x,y})} & F(y\otimes x).
\end{tikzcd}
\end{equation}

\begin{ex}
  The scalar extension functors on categories of modules or algebras are symmetric
  monoidal.
\end{ex}

A "symmetric monoidal natural transformation" is nothing but a monoidal natural transformation
between two symmetric monoidal functors (there is no extra condition involving the symmetric
structure). Therefore we can drop the adjective "symmetric" in that case.

The composition of symmetric monoidal functors is symmetric monoidal (there is again a
$2$-category of symmetric monoidal categories). We adapt the earlier notation, and write
$\Hom_{\otimes}^s(\mathbf{C},\mathbf{D})$ (resp. $\LaxHom_{\otimes}^s(\mathbf{C},\mathbf{D})$) 
for the full subcategory of $\Hom_{\otimes}(\mathbf{C},\mathbf{D})$ (resp.
$\LaxHom_{\otimes}(\mathbf{C},\mathbf{D})$) consisting of the symmetric monoidal functors 
(provided that $\mathbf{C}$ and $\mathbf{D}$ have a symmetric structure, of course).

\begin{rem}
  Given a symmetric strong monoidal functor between two symmetric monoidal categories, if 
  it admits a quasi-inverse, then it admits a \emph{symmetric} monoidal quasi-inverse.
\end{rem}

The set of functions from any set $X$ to some monoid $M$ has
an obvious pointwise monoid structure, but if $X$ is a monoid itself then the subset
of monoid morphisms is in general \emph{not} a submonoid, unless $M$ is commutative. Likewise,
for any (small) category $\mathbf{C}$ and any (small) monoidal category $\mathbf{D}$, the category 
of functors $\mathrm{Fun}(\mathbf{C},\mathbf{D})$ has a natural "pointwise" monoidal structure, 
with unit the constant functor to the unit object of $\mathbf{D}$. If $\mathbf{C}$ is itself
monoidal, then there is a natural forgetful functor
$\Hom_{\otimes}(\mathbf{C},\mathbf{D})\to \mathrm{Fun}(\mathbf{C},\mathbf{D})$, but
$\Hom_{\otimes}(\mathbf{C},\mathbf{D})$ is not monoidal; however, if $\mathbf{D}$ is symmetric monoidal, then 
$\Hom_{\otimes}(\mathbf{C},\mathbf{D})$ becomes a (symmetric) monoidal category, $\mathrm{Fun}(\mathbf{C},\mathbf{D})$
is naturally symmetric, and
$\Hom_{\otimes}(\mathbf{C},\mathbf{D})\to \mathrm{Fun}(\mathbf{C},\mathbf{D})$ is symmetric monoidal.
To be precise, given two monoidal functors $F$ and $G$ in $\Hom_{\otimes}(\mathbf{C},\mathbf{D})$,
the functor $F\otimes G: x\mapsto F(x)\otimes G(x)$ is always well-defined in $\mathrm{Fun}(\mathbf{C},\mathbf{D})$
(and this does not use that $F$ and $G$ are monoidal),
but when $\mathbf{D}$ is symmetric, we give it a monoidal structure by
\[ (F(x)\otimes G(x))\otimes (F(y)\otimes G(y)) \Isom 
(F(x)\otimes F(y))\otimes (G(x)\otimes G(y)) \to F(x\otimes y)\otimes G(x\otimes y) \]
where we use the symmetric structure in the first arrow. Therefore we may speak about
the monoidal functor $F\otimes G\in \Hom_{\otimes}(\mathbf{C},\mathbf{D})$. Note that 
$\Hom_{\otimes}^s(\mathbf{C},\mathbf{D})$ is also then a (symmetric) monoidal subcategory
of $\Hom_{\otimes}(\mathbf{C},\mathbf{D})$.

\subsection{Strong symmetry, inverses and torsion}\label{sec_tors}

For any monoid $M$, the set of monoid morphisms $N\to M$ has a special interpretation
when $N=\N$, $\Z$ or $\Z/n\Z$: it gives respectively $M$, the set $M^\times$ of invertible
elements, and the set $M[n]\subset M^\times$ of invertible elements of order dividing $n$.
When $M$ is commutative, those identifications are compatible with the natural monoid
structure on the set of morphisms $N\to M$. 

We extend those considerations to monoidal categories; an additional layer of difficulty
comes from the fact that one may consider either monoidal or \emph{symmetric} monoidal
functors.

\begin{prop}
  Let $(\mathbf{C},\otimes,I)$ be a monoidal category. Then the canonical functor
  $\Hom_{\otimes}(\langle \N\rangle, \mathbf{C})\to \mathbf{C}$ defined by 
  $F\mapsto F(1)$ is an equivalence of categories. A quasi-inverse is given by sending
  $x\in \mathbf{C}$ to the obvious monoidal functor $\Phi_x:\langle \N\rangle\to \mathbf{C}$
  such that $\Phi_x(n)=x^{\otimes n}$.

  Furthermore, when $\mathbf{C}$ is symmetric, this becomes a symmetric monoidal
  equivalence if $\Hom_{\otimes}(\langle \N\rangle, \mathbf{C})$ is endowed with its
  natural symmetric monoidal structure.
\end{prop}

\begin{proof}
  The only thing to check for the equivalence is that if $F: \langle \N\rangle\to \mathbf{C}$ 
  is a strong monoidal functor, then $\Phi_{F(1)}$ is monoidally isomorphic to $F$, which is 
  immediate from the axioms of monoidal functors.

  The fact that this equivalence is symmetric monoidal when $\mathbf{C}$ is symmetric is clear given
  the definition of the symmetric monoidal structure on 
  $\Hom_{\otimes}(\langle \N\rangle, \mathbf{C})\to \mathbf{C}$.
\end{proof}

\subsubsection*{Strong symmetry}

If $\mathbf{C}$ is symmetric and we want to characterize similarly 
$\Hom_{\otimes}^s(\langle \N\rangle, \mathbf{C})$, we need a definition:

\begin{defi}\label{def_strong}
  Let $(\mathbf{C},\otimes,I)$ be a symmetric monoidal category. An object
  $x\in \mathbf{C}$ is called strongly symmetric if the switching map $s_{x,x}$
  is the identity of $x\otimes x$. 
  
  We write $\mathbf{C}^{ss}$ for the full subcategory of $\mathbf{C}$ consisting of 
  the strongly symmetric elements; we say that $\mathbf{C}$ is strongly symmetric
  if $\mathbf{C}^{ss}=\mathbf{C}$.
\end{defi}

The definition is equivalent to requiring that the canonical group morphism
$\mathfrak{S}_n\to \Aut_{\mathbf{C}}(x^{\otimes n})$ is trivial for $n=2$
(and thus for all $n$). 

\begin{ex}
  If $M$ is a commutative monoid, then $\langle M\rangle$ is strongly symmetric.
\end{ex}

\begin{lem}\label{lem_isom_ss}
  Let $(\mathbf{C},\otimes,I)$ be a symmetric monoidal category. If $x\in \mathbf{C}$
  is isomorphic to a strongly symmetric element, it is strongly symmetric.
\end{lem}

\begin{proof}
  Let $f: x\Isom y$ with $y$ strongly symmetric.
  Since the switching morphism is natural in each variable, the following
  diagram commutes:
  \[ \begin{tikzcd}
    x\otimes x \rar{s_{x,x}} \dar[swap]{f\otimes f} & x\otimes x \dar{f\otimes f} \\
    y\otimes y \rar{s_{y,y}} & y\otimes y.
  \end{tikzcd} \]
  Since $s_{y,y}$ is the identity, so is $s_{x,x}$.
\end{proof}

\begin{lem}\label{lem_im_ss}
  Let $M$ be a commutative monoid, and $\mathbf{C}$ a symmetric monoidal category. If
  $F\in \Hom_{\otimes}^s(\langle M\rangle, \mathbf{C})$, then for any $x\in M$ the object
  $F(x)\in \mathbf{C}$ is strongly symmetric.
\end{lem}

\begin{proof}
  Since $F$ is symmetric, we get the commutative diagram (\ref{diag_sym_funct})
  (with $y=x$). The vertical arrows are equal isomorphisms because $F$ is strongly symmetric,
  and the bottom horizontal arrow is the identity because $x$ is strongly symmetric;
  therefore, the top horizontal arrow is also the identity, which exactly means
  that $F(x)$ is strongly symmetric.
\end{proof}

When $M=\N$, there is a form of converse:

\begin{prop}
  Under the equivalence $\Hom_{\otimes}(\langle \N\rangle, \mathbf{C})\simeq \mathbf{C}$,
  the essential image of the subcategory $\Hom_{\otimes}^s(\langle \N\rangle, \mathbf{C})$
  is exactly $\mathbf{C}^{ss}$.
\end{prop}

\begin{proof}
  If $x\in \mathbf{C}^{ss}$, then the canonical $\Phi_x: n\mapsto x^{\otimes n}$ is clearly
  symmetric, so $x$ is in the essential image. Conversely, if $x$ is in the essential
  image, it is isomorphic to some $F(1)$, which is strongly symmetric by Lemma \ref{lem_im_ss},
  so $x$ is strongly symmetric by Lemma \ref{lem_isom_ss}.
\end{proof}

In particular, we see that $\mathbf{C}^{ss}$ is a \emph{monoidal} subcategory of
$\mathbf{C}$ (which can be shown directly).

\subsubsection*{Inverses}

\begin{defi}
  Let $\mathbf{C}$ be a monoidal category. We define its category of inverses $\mathbf{C}^\times$
  as $\Hom_{\otimes}(\langle\Z\rangle, \mathbf{C})$. If $\mathbf{C}$ is symmetric, its category 
  of symmetric inverses $\mathbf{C}^{\times,s}$ is $\Hom_{\otimes}^s(\langle\Z\rangle, \mathbf{C})$.
\end{defi}

There is an obvious functor $\mathbf{C}^\times\to \mathbf{C}$ sending $F$ to $F(1)$.
Clearly, if $F\in \mathbf{C}^\times$, $x=F(1)$ and $\bar{x}=F(-1)$
must be "weak inverses", in the sense that $x\otimes \bar{x}\simeq \bar{x}\otimes x\simeq I$.
If one unfolds the definition of $F$, one realizes it is exactly the data of an "adjoint
equivalence" $(x,\bar{x},\phi,\psi)$ where $\phi: x\otimes \bar{x}\isom I$ and 
$\psi: \bar{x}\otimes x\isom I$ fit into some commutative diagrams:
\[ \begin{tikzcd}
  x\otimes \bar{x}\otimes x  \dar[swap]{\Id_x\otimes \psi} 
  \rar{\phi\otimes \Id_x} & I \otimes x \dar  \\
  x\otimes I \rar & x
\end{tikzcd}
\quad \text{and} \quad 
\begin{tikzcd}
  \bar{x}\otimes x\otimes \bar{x}  \dar[swap]{\psi\otimes \Id_{\bar{x}}}   
  \rar{\Id_{\bar{x}}\otimes \phi} &  I \otimes x \dar\\
  \bar{x}\otimes I \rar & \bar{x}.
\end{tikzcd} \]

It turns out that it is enough to satisfy one of them (\cite{Bae}). This kind of data
is well-known in the literature, and can be considered as a sort of "coherent inverse"
of $x$. In particular, it follows immediately that the choice of a functor 
$\mathbf{C}\to \mathbf{C}^\times$ such that the composition 
$\mathbf{C}\to \mathbf{C}^\times\to \mathbf{C}$ is isomorphic to the identity
(which of course can only exist if all objects in $\mathbf{C}$ are weakly invertible)
is equivalent to the choice of a "group structure" (or "gs-category" structure) on 
$\mathbf{C}$ in the sense of \cite{Lap} (see also \cite{Ulb}). It corresponds to a coherent
(functorial) choice of inverses for all objects. In that spirit, $\mathbf{C}^\times$ is
a category which embodies all possible such choices, in a canonical way.

Similarly, if $\mathbf{C}$ is symmetric, a symmetric strong monoidal functor 
$\mathbf{C}\to \mathbf{C}^{\times,s}$ such that the composition 
$\mathbf{C}\to \mathbf{C}^{\times,s}\to \mathbf{C}$ is monoidally isomorphic to the identity
is the same as an "abelian gs-category" structure in the sense of \cite{Lap}. Laplaza
shows the following in \cite[4.6,4.7]{Lap}:

\begin{thm}[Laplaza]\label{thm_inverses}
  Let $\mathbf{C}$ be a monoidal category with every object weakly invertible. Then
  if we choose for each $x\in \mathbf{C}$ a weak inverse $\bar{x}$ and an isomorphism
  $x\otimes \bar{x}\isom I$, there is a unique way to extend that to a gs-category structure
  $\mathbf{C}\to \mathbf{C}^\times$.

  Furthermore, it defines an abelian gs-structure $\mathbf{C}\to \mathbf{C}^{\times,s}$
  if and only if $\mathbf{C}$ is strongly symmetric.
\end{thm}


\subsubsection*{Torsion}

\begin{defi}
  Let $\mathbf{C}$ be a symmetric monoidal category. Then for any $n\pgq 2$ we define 
  the symmetric monoidal category $\mathbf{C}[n]$ as 
  $\Hom_{\otimes}^s(\langle\Z/n\Z\rangle, \mathbf{C})$. We call it the $n$-torsion
  category of $\mathbf{C}$.
\end{defi}

Of course we could have looked at the intermediary step where we only consider 
$\Hom_{\otimes}(\langle\Z/n\Z\rangle, \mathbf{C})$ for any monoidal $\mathbf{C}$, but
we won't need it, and already in the case of ordinary groups the notion of $n$-torsion
is much better behaved in a commutative context.

To any $F\in \mathbf{C}[n]$ is associated a choice of isomorphism 
$\phi_x: x^{\otimes n}\isom I$ satisfying some coherence conditions, where $x=F(1)$. 
In particular, $x$ must have weak $n$-torsion, meaning that $x^{\otimes n}\simeq I$.

As before, there is a natural symmetric strong monoidal functor $\mathbf{C}[n]\to \mathbf{C}$,
and we are interested in the choice of a symmetric strong monoidal functor 
$\mathbf{C}\to \mathbf{C}[n]$ such that $\mathbf{C}\to \mathbf{C}[n]\to \mathbf{C}$ is 
monoidally isomorphic to the identity. We call such a choice a "coherent $n$-torsion" structure
on $\mathbf{C}$. Clearly it can only exist if $\mathbf{C}$ is strongly symmetric and
all elements have weak $n$-torsion. Unfortunately, this is no longer sufficient, but
we do have:

\begin{prop}\label{prop_torsion}
  Let $\mathbf{C}$ be a symmetric monoidal category. 
  If $x\in \mathbf{C}$ is strongly symmetric, any choice of isomorphism
  $\phi_x: x^{\otimes n}\isom I$ extends to a unique $F\in \mathbf{C}[n]$.

  Furthermore, if $\mathbf{C}$ is strongly symmetric, a coherent $n$-torsion structure
  is equivalent to the choice of such a $\phi_x$ for all $x\in \mathbf{C}$, such that
  for all $x,y\in \mathbf{C}$ we have a commutative diagram
  \begin{equation}\label{diag_tors_1}
  \begin{tikzcd}
    x^{\otimes n}\otimes y^{\otimes n} \arrow[rr] \drar[swap]{\phi_x\otimes \phi_y} & & 
    (x\otimes y)^{\otimes n} \dlar{\phi_{x\otimes y}} \\
    & I & 
  \end{tikzcd} 
  \end{equation}
  and for each morphism $f:x\to y$ in $\mathbf{C}$:
  \begin{equation}\label{diag_tors_2}
  \begin{tikzcd}
    x^{\otimes n} \drar[swap]{\phi_x} \arrow[rr, "f^{\otimes n}"] & & 
    y^{\otimes n} \dlar{\phi_y} \\
    & I. & 
  \end{tikzcd}
  \end{equation}
\end{prop}

\begin{proof}
  It is easy to see that the extension of $\phi_x$ to $F$ is necessarily unique
  (up to isomorphism of course), and that it exists exactly when we have the commutative
  diagram
  \[ \begin{tikzcd}
    x^{\otimes n}\otimes x \arrow[rr] \drar[swap]{\phi_x\otimes \Id_x} & & 
    x\otimes x^{\otimes n} \dlar{\Id_x\otimes \phi_x} \\
    & x & 
  \end{tikzcd} \]
  where the morphism in the top row is given by the identification of
  $x^{\otimes n}\otimes x$ and $x\otimes x^{\otimes n}$ with $x^{\otimes n+1}$.
  But, up to the action of the permutation group $\mathfrak{S}_{n+1}$
  on the top row, that diagram always commutes, and since $x$ is strongly symmetric 
  that action is trivial.

  It is clear that the two diagrams in the statement are necessary to get a monoidal
  functor $\mathbf{C}\to \mathbf{C}[n]$. Conversely, if they hold, and if we send
  each $x\in \mathbf{C}$ to the unique $F_x\in \mathbf{C}[n]$ extending $\phi_x$,
  to get a functor $\mathbf{C}\to \mathbf{C}[n]$ we need that each $f:x\to y$
  induces a monoidal transformation $F_x\to F_y$, and since the composition 
  $\mathbf{C}\to \mathbf{C}[n]\to \mathbf{C}$ should be isomorphic to the identity
  the component $F_x(1)\to F_y(1)$ has to correspond to $f$, so it is entirely determined,
  and is only well-defined if the diagram (\ref{diag_tors_2}) holds. It is straightforward
  that the diagram (\ref{diag_tors_1}) will guarantee that the functor 
  $\mathbf{C}\to \mathbf{C}[n]$ is monoidal.
\end{proof}

Note that if $\mathbf{C}$ is a groupoid (so all morphisms are invertible), then it is
enough to fix one $\phi_x$ to determine the whole structure (if it exists).

\subsection{Lax monoidal functors and graded rings}\label{sec_grad}

Up until now we have only considered \emph{strong} monoidal functors, but
lax monoidal functors are also key to our central construction. Recall the definition
of graded algebras:

\begin{defi}
  Let $M$ be a commutative monoid and let $R$ be a commutative ring. An $M$-graded $R$-module
  is an $R$-module $A$ endowed with an $R$-module decomposition $A = \bigoplus_{x\in M}A_x$.
  An $M$-graded $R$-algebra is then an algebra which is graded as a module
  such that $1\in A_0$ and $A_x\cdot A_y\subset A_{x+y}$ (writing $M$ additively).
\end{defi}

In the literature, the most common cases are $M=\N$ and $M=\Zd$. Note that if $A$ 
is a graded algebra, $A_0$ is always a ring,
and if $A$ is commutative then $A$ is a graded $A_0$-algebra.
There is an obvious notion of graded module/algebra morphisms, which are
just morphisms preserving the decomposition, and therefore we get
categories $\mcat_M$, $\acat_M$ and $\cacat_M$. These categories are actually 
themselves naturally monoidal, using the tensor product
\[ (A\otimes B)_x = \bigoplus_{y+z=x} A_y\otimes_R B_z, \]
and they have a natural symmetric structure which consists in applying the
switching morphism to each $A_y\otimes_R B_z$.

\begin{ex}
  When $M$ is the trivial monoid, then $\mcat_M$, $\acat_M$ and $\cacat_M$
  are nothing but $\mcat$, $\acat$ and $\cacat$, with their usual
  symmetric monoidal structure.
\end{ex}

Then we can make the following fundamental observation: let $M$ and $N$ be 
commutative monoids; to any \emph{lax} monoidal functor $F: \cat{M}\to \mcat_N$,
we can associate the $R$-module $A_F = \bigoplus_{x\in M} F(x)$, which is
an $(M\times N)$-graded module (since each $F(x)$ is itself $N$-graded).
Then we can define an $R$-bilinear product $A_F\otimes_R A_F\to A_F$ using the maps
$F(x)\otimes F(y)\to F(x+y)$ corresponding to the monoidal structure on $F$.

\begin{prop}
  Let $M$ be a monoid. Then for any lax monoidal functor $F: \cat{M}\to \mcat_N$, 
  the associated graded object $A_F$ is an $(M\times N)$-graded $R$-algebra, 
  and the association $F\mapsto A_F$ defines an equivalence of categories
  \[ \LaxHom_\otimes(\cat{M},\mcat_N) \Isom \acat_{M\times N}. \]

  This also induces an equivalence
  \[ \LaxHom_\otimes^s(\cat{M},\mcat_N) \Isom \cacat_{M\times N}. \]
\end{prop}

\begin{proof}
  The fact that $A_F$ is a graded algebra is exactly a reformulation of the fact
  that $F$ is monoidal: the associativity in $A_F$ corresponds to the compatibility
  of $F$ with associators, and the fact that the structural morphism $R\to F(0)$
  gives a unit element corresponds to the compatibility with unitors.

  There is an obvious quasi-inverse sending a graded algebra $A$ to the functor
  $F:x \mapsto A_x$, which is clearly monoidal for the same reasons.

  The fact that $F$ is symmetric exactly means that $F(x)\otimes F(y)\to F(x+y)$
  and $F(y)\otimes F(x)\to F(x+y)$ are related by the switching morphism
  $F(x)\otimes F(y)\to F(y)\otimes F(x)$, which precisely means that the algebra
  is commutative.
\end{proof}

This gives us the following construction: if $\mathbf{C}$ is any monoidal category,
and $K: \mathbf{C}\to \mcat_N$ is a lax monoidal functor, then the composition
with $K$ defines a functor
\begin{equation*}
  \Hom_{\otimes}(\cat{M}, \mathbf{C}) \xrightarrow{K\circ - } 
  \LaxHom_\otimes(\cat{M}, \mcat_N) \simeq \acat_{M\times N},
\end{equation*}  
and when $K$ is symmetric, we get
\begin{equation*}
  \Hom_{\otimes}^s(\cat{M}, \mathbf{C}) \xrightarrow{K\circ - } 
  \LaxHom_\otimes^s(\cat{M}, \mcat_N) \simeq \cacat_{M\times N}.
\end{equation*}  

Actually, we get something a little better: the objects we get are graded $K(I)$-algebras,
since any commutative graded algebra $A$ is in fact an $A_0$-algebra. Let us
emphasize what we really want to use in the end:

\begin{coro}\label{cor_graded_functors}
  Let $\mathbf{C}$ be a symmetric monoidal category, $N$ a commutative monoid, 
  and let  $K: \mathbf{C}\to \mcat[\Z]_N$ be a symmetric lax monoidal functor.
  Then if $\mathbf{C}$ has an abelian gs-structure, composition with $K$ induces a
  functor
  \[ \mathbf{C}\to \cacat[K(I)]_\Z, \]
  and if $\mathbf{C}$ has a coherent $n$-torsion structure it induces a functor
  \[ \mathbf{C}\to \cacat[K(I)]_{\Z/n\Z}. \]
\end{coro}

\begin{proof}
  The discussion just above shows that from $K$ we get functors
  \[ \mathbf{C}^{\times,s}\To \cacat[K(I)]_\Z,\, \mathbf{C}[n]\To 
  \cacat[K(I)]_{\Z/n\Z}, \]
  simply using $M=\Z$ and $M=\Z/n\Z$. Now we just compose with the structural
  functors $\mathbf{C}\to \mathbf{C}^{\times, s}$ and $\mathbf{C}\to \mathbf{C}[n]$,
  respectively.
\end{proof}

\section{The hermitian Brauer 2-group}\label{sec_brau}

In this section we review hermitian Morita theory, as developped in \cite{FME}
or \cite{Knu}, in the case of central simple algebras with involution
(for which we take \cite{BOI} as a reference). We adopt a categorical
point of view that allows the theory to be expressed in a very efficient way.

\subsection{Hermitian modules and involutions}\label{sec_herm}

We start, for the reader's convenience as well as for establishing notations,
by reviewing basic facts about hermitian modules (see \cite{BOI} for a reference,
or \cite{Knu} for an account over general rings with involution).

\subsubsection*{Morita equivalence}

Let $A$ and $B$ be central simple algebras over $K$. We say that a $B$-$A$-bimodule
$V$ is a Morita equivalence if the following equivalent conditions hold (see \cite[1.10]{BOI}):
\begin{itemize}
  \bitem the left action of $B$ gives an isomorphism $B\simeq \End_A(V)$;
  \bitem the right action of $A$ gives an isomorphism $A\simeq \End_B(V)$.
\end{itemize}
We use the notation $B\mor{V} A$ to state that $V$ is such a bimodule.
Then $A$ and $B$ are Brauer-equivalent iff there exists a Morita equivalence,
which is then unique up to isomorphism.

We can always consider $A$ as a tautological $A$-$A$-bimodule, and it defines a Morita
equivalence. We will often write $|A|$ when we consider $A$ as a vector space or a
module, so we can write $A\mor{|A|} A$.

\begin{ex}
  It is a defining property of Azumaya algebras that the natural "sandwich" map
  \begin{equation}\label{eq_sandwich}
    \anonfoncdef{A\otimes_K A^{op}}{\End_K(|A|)}{a\otimes b}{(x\mapsto axb)}
  \end{equation}
  is a $K$-algebra isomorphism, so we get $(A\otimes_K A^{op})\mor{|A|} K$.
\end{ex}

\subsubsection*{Hermitian forms}

Recall that if $(A,\sigma)$ is an algebra with involution over $(K,\iota)$ 
and $\eps\in U(K,\iota)$, an $\eps$-hermitian module $(V,h)$ over $(A,\sigma)$ 
is a (right) $A$-module $V$ equipped with an $\eps$-hermitian form $h$, meaning that
\[ h: V\times V\to A \]
is bi-additive and satisfies for all $x,y\in V$ and all $a,b\in A$:
\begin{align*}
  h(xa,yb) &= \sigma(a)h(x,y)b, \\
  h(y,x) &= \eps \sigma(h(x,y)).
\end{align*}

We always assume that $\eps$-hermitian forms are regular, meaning that the
induced map $V\to \Hom_A(V,A)$ given by $x\mapsto h(x,-)$ is bijective. An
isometry between two $\eps$-hermitian modules is a module isomorphism which
preserves the hermitian forms. We call $\eps$ the \emph{sign} of $h$, and use
the notation $\eps_h$ to refer to it when it is not already introduced.

\begin{ex}
  An $\eps$-hermitian module over $(K,\Id)$ is simply an $\eps$-symmetric bilinear 
  module over $K$ (and in that case $\eps=\pm 1$).
\end{ex}

Let $(A,\sigma)$ be an algebra with involution over $K$, and let
$a\in \Sym^\eps(A^\times,\sigma)$ (for instance, $a\in K^\times$ and $\eps=1$).
Then we define an $\eps$-hermitian form over $|A|$, by:
\[ \foncdef{\fdiag{a}_\sigma}{|A|\times |A|}{A}{(x,y)}{\sigma(x)ay.} \]
We call such a form \emph{elementary diagonal}. We will write $\fdiag{a_1,\dots,a_n}_\sigma$
for an orthogonal sum $\fdiag{a_1}_\sigma \perp\dots\perp \fdiag{a_n}_\sigma$ (where all
$a_i$ are $\eps$-symmetric), and call such a form \emph{diagonal}.

\begin{rem}
  Not every form is diagonalizable. In fact, for $(V,h)$ to be diagonalizable
  it is clearly necessary that $V$ is a free $A$-module (which is in this case
  equivalent to $\dim_K(V)$ being a multiple of $\dim_K(A)$). It
  turns out that this condition is actually sufficient, except in the special
  case where $(A,\sigma)=(K,\Id)$ and $\eps=-1$, which is the only case where
  $h(x,x)=0$ for all $x\in V$.
\end{rem}

\subsubsection*{Hermitian Morita equivalences}

If $(V,h)$ is an $\eps$-hermitian module over $(A,\sigma)$, then the adjoint
involution $\sigma_h$ on $\End_A(V)$ is defined (see \cite[4.1]{BOI}) by the fact
that for all $x,y\in V$ and all $f\in \End_A(V)$:
\[ h(f(x),y) = h(x,\sigma_h(f)(y)). \]
If $(A,\sigma)$ and $(B,\tau)$ are algebras with involution over $(K,\iota)$, 
we say that $(V,h)$ is an $\eps$-hermitian Morita equivalence between $(B,\tau)$ and
$(A,\sigma)$, which we write
\[ (B,\tau) \mor{(V,h)} (A,\sigma),  \]
if $B\mor{V} A$, $h$ is an $\eps$-hermitian form over $(A,\sigma)$ on $V$, 
and $\tau$ corresponds to $\sigma_h$ through the natural isomorphism $B\simeq \End_A(V)$. 
In particular, any $\eps$-hermitian module $(V,h)$ over $(A,\sigma)$ defines
an equivalence $(\End_A(V),\sigma_h) \mor{(V,h)} (A,\sigma)$.
An isomorphism of $\eps$-hermitian Morita equivalences is a bimodule isomorphism which 
is also an isometry.

If $B\mor{V} A$, then there always exists an $\eps$-hermitian
form $h$ on $V$ such that $(B,\tau) \mor{(V,h)} (A,\sigma)$ (see \cite[4.2]{BOI}). 
When the involutions are unitary (so $\iota\neq \Id$), we can take any $\eps\in U(K,\iota)$;
when the involutions are of the first kind (so $\iota=\Id$), then we must take
$\eps=1$ if $\sigma$ and $\tau$ have the same type (orthogonal or symplectic), 
and $\eps=-1$ otherwise. Moreover, in any case, if $h$ and $h'$ are two possible
choices, then there is $\lambda\in K^\times$ such that $h'=\fdiag{\lambda}h$.

\begin{ex}
  If $a\in \Sym^\eps(A^\times, \sigma)$, then $(A,\sigma_a)\mor{(|A|,\fdiag{a}_\sigma)} 
  (A,\sigma)$, where $\sigma_a(x) = a^{-1}\sigma(x)a$. In particular, if $a\in K^\times$ then
  $\sigma_a=\sigma$.
\end{ex}

\begin{ex}\label{ex_inv_trace}
  Using the involution $\sigma$ on $A$, we can twist the usual sandwich map
  (\ref{eq_sandwich}) to
  \begin{equation}\label{eq_twisted_action}
    \anonfoncdef{A\otimes_K A^{\iota}}{\End_K(|A|)}{a\otimes b}{(x\mapsto ax\sigma(b))}
  \end{equation}
  where $A^{\iota}$ is the same ring as $A$, but with the twisted $K$-algebra
  structure $K\xrightarrow{\iota} K\to A$.

  We call this action of $A\otimes_K A^{\iota}$ on $|A|$ the ``twisted sandwich action''
  (relative to $\sigma$). We will write $|A|_\sigma$ instead of $|A|$ when we
  see it as a left $A\otimes_K A^{\iota}$-module with this action. It is shown in
  \cite[11.1]{BOI} that the so-called involution trace form
  \begin{equation}\label{eq_forme_trace}
    \foncdef{T_\sigma}{|A|\times |A|}{K}{(x,y)}{\Trd_A(\sigma(x)y)}
  \end{equation}
  is a hermitian form over $(K,\iota)$ on the $K$-vector space $|A|$, such that 
  \begin{equation}
    (A\otimes_K A^{\iota},\sigma\otimes \sigma) \mor{(|A|_\sigma,T_\sigma)} (K,\iota).
  \end{equation}
\end{ex}

\subsubsection*{Tensor products}

If $(V_i,h_i)$ is an $\eps_i$-hermitian module over an algebra with involution
$(A_i,\sigma_i)$, for $i=1,2$, then $(V_1\otimes_K V_2,h_1\otimes h_2)$ is an
$\eps_1\eps_2$-hermitian module over $(\mbox{$A_1\otimes_K A_2$},\\
\mbox{$\sigma_1\otimes\sigma_2$})$. In particular, $\fdiag{a}_{\sigma_1}\otimes
\fdiag{b}_{\sigma_2}\simeq \fdiag{a\otimes b}_{\sigma_1\otimes \sigma_2}$.

If for $i=1,2$ we have $(B_i,\tau_i) \mor{(V_i,h_i)} (A_i,\sigma_i)$ then
\[ (B_1\otimes_K B_2,\tau_1\otimes\tau_2) \mor{(V_1\otimes_K V_2,h_1\otimes h_2)}
(A_1\otimes_K A_2,\sigma_1\otimes\sigma_2). \]

\subsubsection*{Conjugate form}

If $(B,\tau) \mor{(V,h)} (A,\sigma)$, then we define $(\bar{V},\bar{h})$, 
where $\bar{V}$ is the $A$-$B$-bimodule defined as $V$ as a $K$-vector space, 
with the action
\[ a\cdot v\cdot b = \tau(b)\cdot v\cdot \sigma(a), \]
and $\bar{h}: \bar{V}\times \bar{V}\to B$ is characterized by
\[  \bar{h}(x,y)z = xh(y,z) \]
for all $x,y,z\in V$. 

Then it is easy to see that $(A,\sigma) \mor{(\bar{V},\bar{h})} (B,\tau)$.

\subsection{The category $\CBrhu$}\label{sec_cbrh}

We now define a category $\CBrhu$, which we call the \emph{hermitian Brauer 2-group}
of $(K,\iota)$, such that:
\begin{itemize}
  \bitem the objects are the algebras with involution $(A,\sigma)$ over $(K,\iota)$;
  \bitem the morphisms from $(B,\tau)$ to $(A,\sigma)$ are the isomorphism classes
    of $\eps$-hermitian Morita equivalences between $(B,\tau) \mor{(V,h)} (A,\sigma)$.
\end{itemize}

We will usually identify an $\eps$-hermitian bimodule and its isomorphism class
when no confusion is possible.
If $f:(B,\tau)\to (A,\sigma)$ is a morphism in $\CBrhu$, we will sometimes
write $(V_f,h_f)$ for the corresponding $\eps$-hermitian bimodule. We usually write
$\CBrh$ for $\CBrh[K,\Id]$.


To properly define $\CBrhu$ as a category, we need to specify how to compose
morphisms. Given two morphisms $(C,\theta)\mor{(U,h)} (B,\tau)$ and 
$(B,\tau) \mor{(V,h')} (A,\sigma)$, we define a $C$-$A$-bimodule
$V\circ U = U\otimes_B V$, and a map $h'\circ h: 
(V\circ U)\times (V\circ U) \to A$ by
\begin{equation}\label{eq_composition_h}
  (h'\circ h)(u\otimes v, u'\otimes v') = h'(v, h(u,u')v').
\end{equation}

\begin{prop}\label{prop_composition}
  In the situation just above, if $h$ and $h'$ are respectively $\eps$
  and $\eps'$-hermitian, then $h'\circ h$ is an $\eps\eps'$-hermitian form
  over $(A,\sigma)$, which defines an equivalence
  \[  (C,\theta) \mor{(V\circ U,h'\circ h)} (A,\sigma) \]
  whose isomorphism class depends only on the isomorphism classes of
  $(U,h)$ and $(V,h')$.
\end{prop}

\begin{proof}
  This is a reformulation of \cite[I.8.1]{Knu} in the special case of
  central simple algebras.
\end{proof}

\begin{ex}\label{ex_comp_diag}
  Let $f: (B,\tau)\mor{(V,h)} (A,\sigma)$ be a morphism in $\CBrhu$, and let
  $b\in \Sym^\eps(B^\times, \tau)$, $a\in \Sym^{\eps'}(A^\times, \sigma)$.
  Then the underlying space of $\fdiag{a}_\sigma\circ f\circ \fdiag{b}_\tau$
  is $|B|\otimes_B V\otimes_A |A|$, which is canonically identified with
  $V$, and the $\eps\eps_f\eps'$-hermitian form is:
  \[ \anonfoncdef{V\times V}{A}{(x,y)}{h(xa,by).} \]
  In particular, $f\circ \fdiag{1}_\tau = \fdiag{1}_\sigma\circ f = f$.
\end{ex}

We can then state:

\begin{prop}\label{prop_brh_cat}
  With the composition of morphisms being given by (\ref{eq_composition_h}),
  $\CBrhu$ is a category, and the identity of $(A,\sigma)$ is $(|A|,\fdiag{1}_\sigma)$.
  Moreover, $\CBrhu$ is actually a groupoid, and the inverse of a morphism $(V,h)$
  is $(\bar{V},\eps_h\bar{h})$.
\end{prop}

\begin{proof}
  All statements are rephrasings of classical statements in hermitian Morita theory,
  in the special case of central simple algebras with involution of the first kind.
  The associativity of the composition is proved in \cite[I.8.1.1]{Knu}. The
  statement on identities follows from Example \ref{ex_comp_diag}, and
  the statement about inverses is proved in \cite[I.9.3.4]{Knu}.
\end{proof}


\subsection{Automorphisms and isometries}\label{sec_autom}

We defined a category $\CBrhu$ in which the (iso)morphisms between two algebras
with involution correspond to hermitian Morita equivalences. But obviously there
is a more elementary notion of isomorphism between algebras with involution.
To be precise, we may define a category $\mathbf{AlgInv}(K,\iota)$, with the same objects
as $\CBrhu$, but where the morphisms are $K$-algebra isomorphisms which are 
compatible with the involutions.

Then to any algebra isomorphism $\phi: (B,\tau)\to (A,\sigma)$ we can associate
$(B,\tau) \mor{(|A|_\phi,\fdiag{1}_\sigma)} (A,\sigma)$
where $|A|_\phi=|A|$ as a right $A$-module, and the left action of $B$ is given by
$b\cdot a = \phi(b)a$ for any $b\in B$ and $a\in |A|$. It is easily seen that this is
indeed a hermitian Morita equivalence, and that this defines a functor
\[ \Theta: \mathbf{AlgInv}(K,\iota)\to \CBrhu \]
which is the identity on objects.

Obviously this functor is far from being full, since there are isomorphisms in
$\CBrhu$ between $(B,\tau)$ and $(A,\sigma)$ whenever $A$ and $B$ are Brauer-equivalent.
On the other hand, we can look at the situation where $(A,\sigma)=(B,\tau)$. We call
\[ \Aut_{\mathcal{M}}(A,\sigma) :=  \Aut_{\CBrhu}(A,\sigma) \]
the group of Morita automorphisms of $(A,\sigma)$, and
\[ \Aut_K(A,\sigma) := \Aut_{\mathbf{AlgInv}(K,\iota)}(A,\sigma) \]
the group of algebraic automorphisms of $(A,\sigma)$. By extension, we say that
a morphism in $\CBrhu$ is algebraic if it is in the image of $\Theta$. Then we have
a group morphism
\[ \Theta: \Aut_K(A,\sigma) \To \Aut_{\mathcal{M}}(A,\sigma) \]
and we want to understand the groups involved, as well as the kernel and cokernel.

Recall from \cite[12.14]{BOI} that the subgroup $\Sim(A,\sigma)\subset A^\times$ of
\emph{similitudes} of $(A,\sigma)$ is defined as
\[ \Sim(A,\sigma) = \ens{a\in A}{a\sigma(a) \in k^\times}. \]
When $a\in \Sim(A,\sigma)$, we can define its \emph{multiplier}
$\mu(a)=a\sigma(a)=\sigma(a)a\in k^\times$, and $\mu$ is a group morphism
$\Sim(A,\sigma)\to k^\times$. Then $\Ker(\mu)$ is the subgroup $\Iso(A,\sigma)$
of \emph{isometries} of $(A,\sigma)$, and $\Ima(\mu)=G(A,\sigma)$ is the group
of multipliers of $(A,\sigma)$. The following proposition gives a complete picture
of the situation:

\begin{prop}\label{prop_autom}
  Let $(A,\sigma)$ be an algebra with involution over $(K,\iota)$. We write 
  $N: K^\times\to k^\times$ for the group morphism $x\mapsto x\iota(x)$ (when
  $\iota\neq \Id$ this is the Galois norm of $K/k$), and $\inte(a)\in \Aut_K(A)$
  for the inner automorphism of $A$ induced by some $a\in A^\times$.
  
  Then we have a commutative diagram of groups with exact rows and columns:
  \[  \begin{tikzcd}[column sep=small]
       & 1 \dar & 1 \dar & 1 \dar & 1 \dar &  \\[-2ex]
      1 \rar & U(K,\iota) \rar \dar & K^\times \rar{N} \dar
          & N(K^\times) \rar \dar & 1 \rar \dar & 1 \\
      1 \rar & \Iso(A,\sigma) \rar \dar{\inte} & \Sim(A,\sigma) \rar{\mu} \dar{\inte}
          & K^\times \rar \dar{\lambda\mapsto (|A|,\fdiag{\lambda}_\sigma)} & K^\times/G(A,\sigma) \rar \dar & 1 \\
      1 \rar & \Ker(\Theta) \rar \dar & \Aut_K(A,\sigma) \rar{\Theta} \dar
          & \Aut_{\mathcal{M}}(A,\sigma) \dar \rar & \Cok(\Theta) \rar \dar & 1 \\[-2ex]
       & 1 & 1 & 1 & 1 & 
       \end{tikzcd} \]
  In particular, there is a canonical isomorphism $\Aut_{\mathcal{M}}(A,\sigma)\simeq 
  K^\times/N(K^\times)$, and when $\iota=\Id$ this gives $\Aut_{\mathcal{M}}(A,\sigma)\simeq
  K^\times/K^{\times 2}$.
\end{prop}

\begin{proof}
  To prove that the diagram commutes, the only non-trivial thing to show is that
  for any $a\in \Sim(A,\sigma)$, $\Theta(\inte(a)) \simeq (|A|, \fdiag{\mu(a)}_\sigma)$.
  By definition, $\Theta(\inte(a))$ corresponds to the hermitian $A$-$A$-bimodule
  $(|A|_a,\fdiag{1}_\sigma)$, where $|A|_a=|A|$ as a right $A$-module,
  and the action on the left is given by $x\cdot y = axa^{-1}y$ for all $x\in A$,
  $y\in |A|_a$. Then we easily see that
  \[ \foncdef{f}{|A|_a}{|A|}{y}{a^{-1}y} \]
  is a bimodule isomorphism. Since
  \[ \fdiag{\mu(a)}_\sigma(f(x),f(y)) = \mu(a)\sigma(f(x))f(y) = 
    \mu(a)\sigma(x)\sigma(a)^{-1}a^{-1}y=\sigma(x)y,  \]
  the map $f$ induces an isometry from $\fdiag{1}_\sigma$ to $\fdiag{\mu(a)}_\sigma$,
  which shows that the diagram commutes.

  The exactness of each row is clear. The exactness of the second column is a consequence
  of \cite[12.15]{BOI}. For the third column, the surjectivity comes from the
  fact that any automorphism of $(A,\sigma)$ in $\CBrhu$ has the form $(|A|,h)$,
  and since $h=\fdiag{1}_\sigma$ is a possible choice, any other choice must have
  the form $h=\fdiag{\lambda}_\sigma$. To prove exactness at $K^\times$, we
  must also show that $(|A|,\fdiag{\lambda}_\sigma)$ and $(|A|,\fdiag{1}_\sigma)$
  are isomorphic iff $\lambda\in N(K^\times)$. But any bimodule isomorphism
  $f:|A|\to |A|$ has the form $f(x)=ax$ for some $a\in K^\times$, so if it induces an
  isometry from $\fdiag{\lambda}_\sigma$ to $\fdiag{1}_\sigma$, we must have 
  $\lambda=a\iota(a)$.

  For the first and fourth columns, the only thing to show is the surjectivity,
  but it is easily obtained by a diagram chase now that we have shown exactness
  everywhere else.
\end{proof}

\begin{rem}\label{ref_alg_map}
  Given two morphisms $f_i: (B_i,\tau_i)\mor{(V_i,h_i)} (A,\sigma)$ in $\CBrhu$, 
  since $\CBrhu$ is a groupoid there is a unique $g: (B_1,\tau_1)\to (B_2,\tau_2)$
  such that $f_1 = f_2\circ g$. Now $h_i$ is isometric to $h_2$ if and only if
  $g$ is algebraic. Thus $\CBrhu$ cannot detect isometric modules on its own,
  but it can if we add the data of $\Theta$.
\end{rem}

\subsection{Monoidal structure and the Goldman element}\label{sec_gold}

The tensor product of algebras with involution endows $\mathbf{AlgInv}(K,\iota)$
with a natural symmetric monoidal structure, with unit object $(K,\iota)$, such that 
the obvious forgetful functor $\mathbf{AlgInv}(K,\iota)\to \acat[K]$ (which is faithful 
but not full) is symmetric strong monoidal.

Likewise, the tensor product of algebras with involution (and of hermitian modules,
for the morphisms) defines a monoidal structure on $\CBrhu$, with unit object $(K,\iota)$, 
such that $\Theta: \mathbf{AlgInv}(K,\iota)\to \CBrhu$ is a strict monoidal functor.
We can also transport the symmetric structure:

\begin{prop}
  The monoidal category $\CBrhu$ has a unique symmetric structure such that
  $\Theta$ is symmetric.
\end{prop}

\begin{proof}
  If we want $\Theta$ to be symmetric we need to define the switching map in $\CBrhu$
  as 
  \[ (A,\sigma)\otimes_K (B,\tau) \xrightarrow{\Theta(s)} (B,\tau)\otimes_K (A,\sigma), \]
  where $s: (A,\sigma)\otimes_K (B,\tau)\to (B,\tau)\otimes_K (A,\sigma)$ is the
  switching map in $\mathbf{AlgInv}(K,\iota)$.

  Since the switching maps, the associators, and the unit maps are all images
  of morphisms in $\mathbf{AlgInv}(K,\iota)$ by $\Theta$, all the coherence axioms
  can be verified in $\mathbf{AlgInv}(K,\iota)$.
\end{proof}

In particular, this means that for any algebra with involution $(A,\sigma)$
over $(K,\iota)$ and any $n\in \N$ there are natural group morphisms
\[ \begin{tikzcd}
  \mathfrak{S}_n \rar \drar & \Aut_K(A^{\otimes n},\sigma^{\otimes n}) \dar{\Theta} \\
  & \Aut_{\mathcal{M}}(A^{\otimes n},\sigma^{\otimes n}).
\end{tikzcd} \]

A remarkable feature of Azumaya algebras is the existence of the so-called
Goldman element (see \cite[3.A]{BOI}).

\begin{defi}
  Let $A$ be a central simple algebra over $K$. Its Goldman element $g_A\in A\otimes_K A$
  is defined by the fact that the sandwich map (\ref{eq_sandwich}) sends $g_A$, seen as
  an element of $A\otimes_K A^{op}$, to the reduced trace $\Trd_A: A\to K$, viewed as a
  linear map $|A|\to |A|$.
\end{defi}

It is shown in \cite[10.1]{BOI} that the natural morphism 
$\mathfrak{S}_n\to \Aut_K(A^{\otimes n})$ admits a lift
\[  \begin{tikzcd}
      & (A^{\otimes n})^\times \dar{\inte} \\
    \mathfrak{S}_n \rar \urar & \Aut_K(A^{\otimes n})
  \end{tikzcd} \]
where $\mathfrak{S}_n\to (A^{\otimes n})^\times$ is uniquely charaterized by
\[ (i, i+1) \mapsto 1\otimes\cdots \otimes g_A \otimes 1\otimes \dots \otimes 1 \]
(with $g_A$ occupying the $i$ and $i+1$ slots).

\begin{prop}\label{prop_sym_iso}
  Let $(A,\sigma)$ be an algebra with involution over $(K,\iota)$. Then the natural
  morphism $\mathfrak{S}_n\to (A^{\otimes n})^\times$ gives a commutative diagram
  of groups:
  \[ \begin{tikzcd}
      & \Iso(A^{\otimes n},\sigma^{\otimes n}) \dar{\inte} \\
    \mathfrak{S}_n \rar \urar & \Aut_K(A^{\otimes n},\sigma^{\otimes n}).
  \end{tikzcd} \]
\end{prop}

\begin{proof}
  The only thing to show is that $\mathfrak{S}_n\to (A^{\otimes n})^\times$ actually
  takes values in $\Iso(A^{\otimes n},\sigma^{\otimes n})$. It is enough to consider
  transpositions of the form $(i,i+1)$, which means that it is enough to show
  that $g_A\in \Iso(A^{\otimes 2},\sigma^{\otimes 2})$. Since $g_A^2=1$ (\cite[3.6]{BOI}),
  this is the same as proving that $g_A$ is symmetric for $\sigma\otimes\sigma$,
  which is proved in \cite[10.19]{BOI}.
\end{proof} 

The crucial consequence is the following:

\begin{coro}\label{cor_strong_sym}
  The symmetric monoidal category $\CBrhu$ is strongly symmetric.
\end{coro}

\begin{proof}
  The morphism $\mathfrak{S}_n \to \Aut_{\mathcal{M}}(A^{\otimes n},\sigma^{\otimes n})$
  factors through $\Iso(A^{\otimes n},\sigma^{\otimes n})$ by Proposition \ref{prop_sym_iso},
  which by Proposition \ref{prop_autom} means that it is trivial.
\end{proof}

\subsection{Inverses and coherent $2$-torsion}\label{sec_tors_cbrh}

For any small category $C$, let us write $\pi_0(C)$ for its set of isomorphism
classes. When $C$ is monoidal, $\pi_0(C)$ inherits a natural monoid structure,
which is commutative when $\mathbf{C}$ is symmetric.
As we have already stated, there is a morphism in $\CBrhu$ between $(B,\tau)$
and $(A,\sigma)$ if and only if $A$ and $B$ are Brauer-equivalent, which means
that we get a canonical embedding $\pi_0(\CBrhu)\to \Br(K)$. In particular,
all objects in $\CBrhu$ are weakly invertible. When $\iota\neq \Id$, $\pi_0(\CBrhu)$
is identified with the kernel of the norm map $\Br(K)\to \Br(k)$; when
$\iota=\Id$, $\pi_0(\CBrh)$ is identified with the $2$-torsion subgroup $\Br(K)[2]$
(in particular, every object has weak $2$-torsion).

In fact, Example \ref{ex_inv_trace} gives a more precise statement about invertibility
and torsion, as it gives an explicit weak inverse of each
$(A,\sigma)$, namely $(A^{\iota},\sigma)$, which is indeed equal to $(A,\sigma)$
when $\iota=\Id$, and an explicit equivalence from
their product to the unit object $(K,\iota)$, given by the involution trace form 
$(|A|_\sigma,T_\sigma)$.
\\

According to Theorem \ref{thm_inverses}, since $\CBrhu$ is strongly symmetric by
Corollary \ref{cor_strong_sym}, the data of all those equivalences defines a canonical
abelian gs-structure $\CBrhu \to \CBrhu^{\times,s}$. This is always the one
we implicitly refer to when we mention such a structure.

When $\iota=\Id$, we can hope to get a coherent $2$-torsion structure since all
objects have weak $2$-torsion, and indeed:

\begin{thm}\label{thm_torsion}
  Associating to each $(A,\sigma)\in \CBrh$ the canonical equivalence 
  $(A\otimes_K A,\sigma\otimes \sigma)\mor{(|A|_\sigma,T_\sigma)} (K,\Id)$ 
  defines, in the way of Proposition \ref{prop_torsion}, a canonical coherent $2$-torsion 
  structure on $\CBrh$.
\end{thm}

\begin{proof}
  According to Proposition \ref{prop_torsion}, we have to verify the two diagrams
  (\ref{diag_tors_1}) and (\ref{diag_tors_2}). The first one is straightforward, using the fact that
  the reduced trace $\Trd_{A\otimes_K B}$ of $A\otimes_K B$ is nothing but
  $\Trd_A\otimes \Trd_B$.

  The second one requires more work; it states that for any 
  $f:(B,\tau)\mor{(V,h)} (A,\sigma)$ in $\CBrh$, the diagram
  \[ \begin{tikzcd} \label{diagr_torsion_commut}
      (B\otimes_K B,\tau\otimes \tau) \ar{rr}{f^{\otimes 2}} \drar[swap]{(B,T_\tau)} &
      & (A\otimes_K A,\sigma\otimes\sigma) \dlar{(A,T_\sigma)} \\
      & I &
    \end{tikzcd} \]
  commutes. We define the following $(B\otimes_K B)$-$K$-bimodule morphism:
  \[ \foncdef{\psi}{(V\otimes_K V)\otimes_{A\otimes_K A}A}{B}
     {(v\otimes w)\otimes a}{\phi_h(va\otimes w).} \]
   where $\phi_h: V\otimes_K V\to B$ (see \cite[5.1]{BOI}) is given, identifying
   $B=\End_A(V)$, by
  \[ \phi_h(v\otimes w)(x) = vh(w,x). \]

  Then $\psi$ is well-defined since for $x,y\in A$:
  \begin{align*}
    \psi((v\otimes w)\otimes (xa\sigma(y)) &= \phi_h(vxa\sigma(y)\otimes w) \\
    &= \phi_h(vxa\otimes wy) \\
    &= \psi ((vx\otimes wy)\otimes a),
  \end{align*}
  and it is a bimodule morphism since for $x,y\in B$:
  \begin{align*}
    \psi((xv\otimes yw)\otimes a) &= \phi_h(xva\otimes yw) \\
    &= x\phi_h(va\otimes v)\tau(y).
  \end{align*}
  
  To show that $\psi$ is an isometry, we must establish equality between
  on the one hand
  \begin{align*}
    & \Trd_B\left(\tau \left( \psi((v\otimes w)\otimes a)\right)\cdot \psi(((v'\otimes w')\otimes b))\right) \\
    &= \Trd_B\left( \tau(\phi_h(va\otimes w))\cdot \phi_h(v'b\otimes w')\right)
  \end{align*}
  and on the other hand
  \begin{align*}
    & \Trd_A \left(\sigma(a)(h\otimes h)(v\otimes w,v'\otimes w')\cdot b\right) \\
    &= \eps \Trd_A(\sigma(a)h(v,v')bh(w',w))
  \end{align*}
  with $\eps=\eps_h$. Now applying successively the formulas in theorem \cite[5.1]{BOI},
  we get:
  \begin{align*}
    & \Trd_B\left( \tau(\phi_h(va\otimes w)) \cdot \phi_h(v'b\otimes w')\right) \\
    &= \eps \Trd_B\left(\phi_h(w\otimes va) \cdot \phi_h(v'b\otimes w')\right) \\
    &= \eps \Trd_B\left(\phi_h( wh(va,v'b) \otimes w') \right) \\
    &= \eps \Trd_A\left(h(w',wh(va,v'b))\right) \\
    &= \eps \Trd_A\left(h(w',w)\sigma(a)h(v,v')b\right). \qedhere
  \end{align*}
\end{proof}

\section{The mixed Witt ring}\label{sec_mixed}

In this section we combine the techniques developped in Section \ref{sec_grad} to define
graded rings with the structure on $\CBrhu$ we established in \ref{sec_tors_cbrh}.

\subsection{The Witt groups}\label{sec_group}

We start by defining the underlying group structures.

\begin{defi}
  Let $(A,\sigma)$ be an algebra with involution over $(K,\iota)$, and let 
  $\eps\in U(K,\iota)$. We denote by $SW^\eps(A,\sigma)$ the set of isometry classes
  of $\eps$-hermitian modules over $(A,\sigma)$; it is a commutative semi-group when 
  equipped with the orthogonal direct sum of $\eps$-hermitian modules.

  We define $SW_\eps(A,\sigma)=SW^{t(\sigma)\eps}(A,\sigma)$ where $t(\sigma)= 1$ when 
  $\sigma$ is orthogonal or unitary, and $t(\sigma)= -1$ when $\sigma$ is symplectic.

  We also set: $SW_\pm(A,\sigma) = SW_1(A,\sigma)\oplus SW_{-1}(A,\sigma)$.
\end{defi}

We often make the slight abuse of notations $SW_\eps(K)$ for $SW_\eps(K,\Id)$;
note that $SW_\eps(K)=SW^\eps(K)$.

We can observe that there is a natural functoriality for those semi-groups: if
$f: (B,\tau) \to (A,\sigma)$ is a morphism in $\CBrhu$ with sign $\eps'\in U(K,\iota)$,
then it induces a function
\begin{equation}\label{eq_induced}
  f_*: SW^{\eps}(B,\tau)\To SW^{\eps\eps'}(A,\sigma)
\end{equation}
given by $g\mapsto f\circ g$, where we see any $(V,h)\in SW^{\eps}(B,\tau)$
as a morphism $g: (\End_K(V),\tau_h)\mor{(V,h)} (B,\tau)$. 

To simplify things a bit, and with no real consequence to the theory, we are
going to consider the subcategory $\CBrhu'$ of $\CBrhu$, which is equal to
$\CBrhu$ when $\iota=\Id$, but only includes the morphisms which are $1$-hermitian
when $\iota\neq \Id$. The restriction is mostly harmless since any two objects
which are isomorphic in $\CBrhu$ are still isomorphic in $\CBrhu'$.

\begin{prop}
  For any $\eps\in U(K,\iota)$, the association $(A,\sigma)\mapsto SW_\eps(A,\sigma)$
  and $f\mapsto f_*$ as described above defines a functor from $\CBrhu'$ to
  the category of commutative semi-groups. In particular, each $f_*$ is actually
  an isomorphism.
\end{prop}

\begin{proof}
  First we need to establish that $f_*$ does define a function from $SW_\eps(B,\tau)$
  to $SW_\eps(A,\sigma)$. When $\iota\neq \Id$, this comes from the restriction to
  $\CBrhu'$, since in (\ref{eq_induced}) we have $\eps'=1$. When $\iota=\Id$, there
  is no a priori restriction on $\eps'$, but it is actually determined by the types
  of $\sigma$ and $\tau$. In fact, we get
  \[ f_*: SW_{t(\tau)\eps}(B,\tau)\To SW_{t(\sigma)\eps\eps'}(A,\sigma) \]
  and since $t(\tau) = \eps't(\sigma)$, the signs do correspond on both sides.

  The functoriality is a direct consequence of the associativity of the composition
  in $\CBrhu$. The fact that it is a semi-group morphism amounts to the formula
  $h\circ (h_1\perp h_2) \simeq (h\circ h_1)\perp (h\circ h_2)$ which can be checked
  directly on the definition of the composition of hermitian forms.

  Since all morphisms in $\CBrhu'$ are invertible, all $f_*$ are invertible too
  by functoriality.
\end{proof}

The point of introducing $SW_\eps$ (instead of the more obvious $SW^\eps$) is precisely
so we get the functoriality in the previous proposition. 


\begin{defi}
  For each $\eps\in U(K,\iota)$ and each algebra with involution $(A,\sigma)$
  over $(K,\iota)$, we define the Grothendieck-Witt group $GW_\eps(A,\sigma$)
  as the Grothendieck group of the semi-group $SW_\eps(A,\sigma)$.

  If $(V,h)\in SW_\eps(A,\sigma)$, we say that a submodule $W\subset V$ is
  a Lagrangian of $(V,h)$ if $W=W^\perp$. If $(V,h)$ has a Lagrangian, it is called
  hyperbolic.

  We define the Witt group $W_\eps(A,\sigma)$ as the quotient of $GW_\eps(A,\sigma)$
  by the subgroup generated by the hyperbolic forms.
\end{defi}

Of course we also have groups $GW^\eps$ and $W^\eps$, but we are less interested in
them because they lack the functoriality we are looking for. We also have in an
obvious way $GW_\pm$ and $W_\pm$.

Note that since by Witt's theorem $SW_\eps(A,\sigma)$
satisfies the cancellation property, it embeds in $GW_\eps(A,\sigma)$.

\begin{prop}
  For any $\eps\in U(K,\iota)$, there is a unique structure of functor from 
  $\CBrhu'$ to the category of abelian groups on $GW_\eps$ and $W_\eps$ such that 
  the canonical maps $SW_\eps\to GW_\eps\to W_\eps$ are natural transformations.
\end{prop}

\begin{proof}
  Since the construction of the Grothendieck group of a semi-group is functorial,
  $GW_\eps$ is the composition of two functors, so it is a functor. To see that $W_\eps$
  is a functor, we just need to show that any $f_*$ sends hyperbolic forms to hyperbolic
  forms, which is clear since if $W\subset V$ is a Lagrangian of $(V,h)$, then
  $f_*(W,h_{|W})$ is easily seen to be a Lagrangian of $f_*(V,h)$.
\end{proof}

Clearly, we get functors $GW_\pm$ and $W_\pm$ from $\CBrhu$ to $\mu_2(K)$-graded 
abelian groups.

\subsection{The mixed Witt ring}\label{sec_ring}

It is easy to see that the natural maps
\[ SW^\eps(A,\sigma) \times SW^{\eps'}(B,\tau)
\to SW^{\eps\eps'}(A\otimes_K B, \sigma\otimes \tau) \]
given by the tensor products of hermitian modules define maps
\begin{equation*}
  SW_\eps(A,\sigma) \times SW_{\eps'}(B,\tau)
\to SW_{\eps\eps'}(A\otimes_K B, \sigma\otimes \tau).
\end{equation*}

This is because $t(\sigma\otimes \tau)=t(\sigma)\cdot t(\tau)$. Since
these maps are additive in each variable, they induce group morphisms
\begin{equation}\label{eq_monoid_gw}
  GW_\eps(A,\sigma) \otimes GW_{\eps'}(B,\tau)
\to GW_{\eps\eps'}(A\otimes_K B, \sigma\otimes \tau)
\end{equation}
by the universal property of the Grothendieck groups.

\begin{prop}
  The maps (\ref{eq_monoid_gw}) define on $GW_\pm$ a structure of symmetric
  lax monoidal functor from $\CBrhu'$ to $\mcat[\Z]_{\mu_2(K)}$. Furthermore, there is 
  a unique structure of symmetric lax monoidal functor on $W_\pm$ such that
  the canonical transformation $GW_\pm\to W_\pm$ is monoidal.
\end{prop}

\begin{proof}
  The fact that $GW_\pm$ is symmetric monoidal is completely straightforward given the
  usual properties of tensor products. To see that it induces a symmetric monoidal
  structure on $W_\pm$, the only point is that the products are well-defined,
  which means that if $(V,h)$ is hyperbolic, then so is $(U\otimes_K V,h'\otimes h)$
  for any $(U,h')$. Now if $W\subset V$ is a Lagrangian for $h$, it is easy to see that
  $U\otimes_K W$ is a Lagrangian for $h'\otimes h$.
\end{proof}

When $\iota\neq \Id$, we are not really interested in $GW_\pm$ and $W_\pm$ but rather
simply in the neutral components $GW_1$ and $W_1$ (since $GW_{-1}\approx GW_1$ in
that case, albeit non-canonically), which give symmetric monoidal functors to 
the category of abelian groups.

We can therefore apply Corollary \ref{cor_graded_functors}. Using
respectively $K=GW_1$ and $K=W_1$ and the canonical abelian gs-structure of $\CBrhu'$ we 
described in \ref{sec_tors_cbrh}, we get functors:
\begin{equation}
  \hat{GW}: \CBrhu' \to \cacat[GW(K,\iota)]_{\Z}
\end{equation}
\begin{equation}
  \hat{W}: \CBrhu' \to \cacat[W(K,\iota)]_{\Z}.
\end{equation}

In concrete terms: 
\[ \hat{GW}(A,\sigma) = \bigoplus_{n\in \N^*} GW_1((A^\iota)^{\otimes n},\sigma^{\otimes n})
\oplus GW_1(K,\iota) \oplus \bigoplus_{n\in \N^*} GW_1(A^{\otimes n},\sigma^{\otimes n}) \]
and likewise for $\hat{W}(A,\sigma)$. An element in 
$GW_1((A^\iota)^{\otimes n},\sigma^{\otimes n})$ has degree $-n$, and when an element 
of positive degree is mutliplied with an element of negative degree, the difference
is canceled using the hermitian Morita equivalence 
$(A\otimes A^\iota,\sigma\otimes \sigma) \mor{} (K,\iota)$ as many times as necessary.

The point of the machinery developed in Section \ref{sec_grad} is that it guarantees that this
is well-defined, and gives a commutative $\Z$-graded ring, functorial in $(A,\sigma)$ (none
of which is trivial). Although this ring has interesting applications for the theory
of hermitian forms over algebras with unitary involution, we will put it aside in
the remainder of the article, and focus on the case of involutions of the first kind,
for which a more powerful construction is available.
\\

If we now specialize to $\iota=\Id$, we can use $K=GW_\pm$ and $K=W_\pm$
in Corollary \ref{cor_graded_functors}, as well as the
canonical coherent $2$-torsion structure of $\CBrh$, to get functors
\begin{equation}
  \tld{GW}: \CBrh \to \cacat[GW(K)]_{\Gamma}
\end{equation}
\begin{equation}
  \tld{W}: \CBrh \to \cacat[W(K)]_{\Gamma}
\end{equation}
where $\Gamma = \Zd \times \mu_2(K)$. 

In concrete terms, for any algebra with involution (of the first kind), we have
the \emph{mixed Grothendieck-Witt ring}
\begin{equation}
  \tld{GW}(A,\sigma) = GW_1(K) \oplus GW_{-1}(K) \oplus GW_1(A,\sigma) \oplus 
  GW_{-1}(A,\sigma)
\end{equation}
and the \emph{mixed Witt ring}
\begin{equation}
  \tld{W}(A,\sigma) = W_1(K) \oplus W_1(A,\sigma) \oplus W_{-1}(A,\sigma)
\end{equation}
(note that $W_{-1}(K)=0$ because all anti-symmetric bilinear forms are hyperbolic).
The word "mixed" is here to indicate that the ring structure mixes bilinear and
hermitian forms. Of course, given its construction, $\tld{W}(A,\sigma)$ is the quotient
of $\tld{GW}(A,\sigma)$ by the ideal of hyperbolic forms.

A lot of the ring structure of $\tld{GW}(A,\sigma)$ comes from the fact that
$GW_\pm(K)$ is a commutative ring and that $GW_\pm(A,\sigma)$ is a module
over this ring, which is not something new. The novel part is the product
\[ GW_\eps(A,\sigma)\otimes GW_{\eps'}(A,\sigma)\to GW_{\eps\eps'}(K), \] 
which simply comes from the natural product $GW_\eps(A,\sigma)\otimes GW_{\eps'}(A,\sigma)
\to GW_{\eps\eps'}(A^{\otimes 2},\sigma^{\otimes 2})$ composed with the isomorphism
$GW_{\eps\eps'}(A^{\otimes 2},\sigma^{\otimes 2})\simeq GW_{\eps\eps'}(K)$ coming
from the canonical hermitian Morita equivalence between $(A^{\otimes 2},\sigma^{\otimes 2})$
and $(K,\Id)$, given by $(|A|_\sigma,T_\sigma)$. Once again, the machinery
of Section \ref{sec_grad} guarantees that everything is well-defined and functorial. (Of course,
similar statements hold for $\tld{W}(A,\sigma)$).

\begin{rem}
  In the case of a quaternion algebra $Q$ with its canonical involution $\gamma$,
  Lewis makes in \cite{Lew} a very similar construction to our $\tld{W}(Q,\gamma)$.
  His definition amounts to essentially the same, except that he uses the norm form
  of $Q$ instead of the involution trace form $T_\gamma$. Since these two forms differ
  by a factor $\fdiag{2}$, we get non-isomorphic but very similar rings. However,
  the norm form is a special feature of quaternion algebras (in general, for an
  algebra of degree $n$ the reduced norm is a homogeneous polynomial function
  of degree $n$), so the construction does not generalize well to arbitrary algebras.
  Furthermore, no proof of the associativity or commutativity of the product 
  is given in \cite{Lew}.
\end{rem}

We will call $GW_\pm(K)$ the \emph{even} part of $\tld{GW}(A,\sigma)$, and 
$GW_\pm(A,\sigma)$ its \emph{odd} part. This corresponds to considering the 
$\Zd$-grading canonically induced by the $\Gamma$-grading on the ring. Then the
functoriality of $\tld{GW}$ behaves very differently on the even and odd part:
for any $f: (B,\tau)\to (A,\sigma)$ in $\CBrh$, the induced $f_*: \tld{GW}(B,\tau) \to
\tld{GW}(A,\sigma)$ consists of the \emph{identity} on the even part, and of
the morphism described in (\ref{eq_induced}) on the odd part. 

For instance, any automorphism of $(A,\sigma)$ in $\CBrh$ induces a graded ring
automorphism on $\tld{GW}(A,\sigma)$; such an automorphism will be called
\emph{standard}. Recall from Proposition \ref{prop_autom} that all Morita automorphisms of
$(A,\sigma)$ are of the form $(|A|,\fdiag{\lambda}_\sigma)$ with $\lambda\in K^\times$,
so the associated standard automorphism of $\tld{GW}(A,\sigma)$ is the identity
on $GW_\pm(K)$, and is the multiplication of hermitian forms by the scalar $\lambda$
on $GW_\pm(A,\sigma)$.

Since all morphisms in $\CBrh$ are invertible, this means that all the induced
morphisms between mixed Witt rings are actually isomorphisms. In particular,
given some algebra with involution $(A,\sigma)$, for any other $(B,\tau)$ such
that $A$ and $B$ are Brauer-equivalent, there are isomorphisms $\tld{GW}(A,\sigma)
\approx \tld{GW}(B,\tau)$, well-defined up to a standard automorphism of 
$\tld{GW}(B,\tau)$. This means that we can translate any computation in $\tld{GW}(A,\sigma)$
to a computation in $\tld{GW}(B,\tau)$; the even component of the result does not depend 
on any choice of equivalence, while the odd part is well-defined up to multiplication 
by a scalar (which corresponds to a choice of equivalence). For instance, it is
often convenient to choose for $B$ a division algebra, so that we can reduce
all computations to the case of diagonal forms.

Of course everything we said in the last paragraphs also holds for $\tld{W}(A,\sigma)$
(for the same reasons, or by seeing it as a quotient of $\tld{GW}(A,\sigma)$).

\begin{ex}
  We have by construction $\tld{GW}(K,\Id) = GW^\pm(K) \oplus GW^\pm(K)$
  and $\tld{W}(K,\Id) = W(K) \oplus W(K)$ as $\Gamma$-graded groups (in that
  second case, the ring does have four components, but two of them are zero).
  It is easy to see by definition of the product that as $\Gamma$-graded
  rings, we have canonical isomorphisms $\tld{GW}(K,\Id)\simeq GW_\pm(K)[\Zd]$
  (where $R[G]$ denotes the group algebra of a group $G$ over a ring $R$)
  and $\tld{W}(K,\Id)\simeq W(K)[\Zd]$ (where $GW_\pm(K)$ and $W(K)$ are
  considered as $\mu_2$-graded rings).

  In particular, if $(A,\sigma)$ is a split algebra with orthogonal involution,
  then there is a ring isomorphism $\tld{W}(A,\sigma)\approx W(K)[\Zd]$, but
  it depends on the choice of a quadratic form $q$ such that $\sigma=\sigma_q$
  (which is only well-defined up to a scalar factor).
\end{ex}

\begin{rem}
  We can also consider the subrings $\tld{GW}_\eps(A,\sigma) = GW(K) \oplus 
  GW_\eps(A,\sigma)$ and $\tld{W}_\eps(A,\sigma) = W(K) \oplus W_\eps(A,\sigma)$
  for $\eps=\pm 1$, which are also functorial over $\CBrh$. The inconvenience
  is that in that case we only consider one sign of hermitian forms at a time,
  and that sign depends on the type of $\sigma$.
\end{rem}

\begin{rem}
  Note that we can associate to each $2$-torsion Brauer class $[A]$ the 
  isomorphism class of the ring $\tld{GW}(A,\sigma)$ where $\sigma$ 
  is any involution on $A$. This is well-defined, but is more or less
  unusable since to actually work with the ring we do need to choose 
  a representative $(A,\sigma)$.
\end{rem}

\subsection{The reduced dimension map}\label{sec_dim}

In the classical theory of quadratic forms, the dimension maps $GW(K)\to \Z$
and $W(K)\to \Zd$ play an important role (especially through their kernel).
We can generalize that to mixed Witt rings: to any $\eps$-hermitian module
$(V,h)$ over $(A,\sigma)$, we can associate its \emph{reduced dimension} 
$\rdim(V)$ (or $\rdim(h)$), defined as the degree of the algebra 
$\End_A(V)$. It can also be characterized as $\dim_K(V)/\deg(A)$ (see \cite{BOI}).

This defines a semi-group morphism $\rdim: SW_\eps(A,\sigma)\to \N$, which extends to
a group morphism $\rdim: GW_\eps(A,\sigma)\to \Z$. If we use these morphisms 
componentwise, we get a $\Gamma$-graded group morphism
\begin{equation}\label{eq_dim_gw}
  \tld{\rdim}: \tld{GW}(A,\sigma)\to \Z[\Gamma].
\end{equation}
Since hyperbolic forms always have an even reduced dimension, this also
induces a group morphism
\begin{equation}\label{eq_dim_w}
  \tld{\rdim}_2: \tld{W}(A,\sigma)\to \Zd[\Gamma].
\end{equation}

\begin{prop}\label{prop_dim}
  The morphisms (\ref{eq_dim_gw}) and (\ref{eq_dim_w}) are actually $\Gamma$-graded
  ring morphisms, and for any $f: (B,\tau)\to (A,\sigma)$ in $\CBrh$, we get
  a commutative diagram
  \[ \begin{tikzcd}
    \tld{GW}(B,\tau) \arrow[ddd] \drar{f_*} \arrow[rr, "\tld{\rdim}"] & & 
    \Z[\Gamma] \arrow[ddd] \\
    & \tld{GW}(A,\sigma) \dar \urar{\tld{\rdim}} & \\
    & \tld{W}(A,\sigma) \drar{\tld{\rdim}_2} & \\
    \tld{W}(B,\tau) \urar{f_*} \arrow[rr, "\tld{\rdim}_2"] & & \Zd[\Gamma].
  \end{tikzcd} \]
\end{prop}

\begin{proof}
  We need to show that $\rdim(V\otimes V') = \rdim(V)\dot \rdim(V')$, which is clear 
  since $\End_{A\otimes_K A'}(V\otimes_K V')$ is isomorphic to 
  $\End_A(V)\otimes_K \End_{A'}(V')$, and the degrees multiply.

  Then the fact that the diagram commutes relies on the fact that the reduced
  dimension is preserved by the morphisms $f_*: GW_\eps(B,\tau)\to GW_\eps(A,\sigma)$
  induced by hermitian Morita equivalences. But if $(V,h)\in SW_\eps(B,\tau)$,
  then by construction of $f_*$, $f_*(V,h)$ defines an equivalence $(\End_B(V),\tau_h)
  \mor{} (A,\sigma)$, so the degree of the algebra on the left-hand side is preserved
  in the operation (and that is the reduced dimension). 
\end{proof}

\subsection{Products of diagonal forms}\label{sec_diag}

Now that we have established the formal properties of our mixed rings, we would like to
be able to perform explicit computations. Obviously the only products which present
any difficulty are the products of two elements in $GW_\pm(A,\sigma)$. Actually,
since the elements in $GW_{-1}(K)$ are all hyperbolic and characterized by their 
(reduced) dimension,  and we know from Proposition \ref{prop_dim} how reduced dimensions 
multiply, the only non-trivial products to compute are those of two elements in
the same $GW_\eps(A,\sigma)$.

Furthermore, it is enough to know how to multiply elementary diagonal forms
$\fdiag{a}_\sigma$, since we know that if needed we can perform these computations in 
$\tld{GW}(D,\theta)$ where $D$ is the division algebra Brauer-equivalent to our
algebra $A$ (and $\theta$ is any involution). The result will even be independant
of the choice of equivalence between $(A,\sigma)$ and $(D,\theta)$ as it lies in
the even component.

In \cite[Â§11]{BOI}, given $a\in \Sym^\eps(A^\times,\sigma)$, a symmetric
bilinear form $T_{\sigma,a}: A\times A\to K$ is introduced, called a twisted involution 
trace form. We generalize that slightly by taking  $a,b\in \Sym^\eps(A^\times,\sigma)$ 
and defining $T_{\sigma,a,b}: A\times A\to K$ as
\begin{equation}
  T_{\sigma,a,b}(x,y) = \Trd_A(\sigma(x)ay\sigma(b)).
\end{equation}

We recover the inital $T_{\sigma,a}$ as $T_{\sigma,a,1}$, and also the usual
involution trace form $T_\sigma$ as $T_{\sigma,1,1}$. Then we get:

\begin{prop}\label{prop_prod_diag}
  Let $(A,\sigma)$ be an algebra with involution over $K$, and let
  $a,b\in \Sym^\eps(A^\times,\sigma)$. Then in $\tld{GW}(A,\sigma)$ we have
  \[ \fdiag{a}_\sigma \cdot \fdiag{b}_\sigma = T_{\sigma,a,b} \in GW(K). \]
\end{prop}

\begin{proof}
  By definition, $\fdiag{a}_\sigma \cdot\fdiag{b}_\sigma$ is the bilinear form given by 
  the composition $T_\sigma\circ \fdiag{a\otimes b}_{\sigma\otimes\sigma}$.
  According to Example \ref{ex_comp_diag}, this is
  \[ (x,y)\mapsto T_\sigma(x, (a\otimes b)\cdot y), \]
  where the action of $A\otimes_K A$ on $A$ is the twisted sandwich action (\ref{eq_twisted_action}).
  Since $T_\sigma(x,y)=\Trd_A(\sigma(x)y)$ and $(a\otimes b)\cdot y = ay\sigma(b)$,
  we may conclude.
\end{proof}

\begin{ex}\label{ex_carre_1}
  In particular, $\fdiag{1}_\sigma^2=T_\sigma$, which of course follows
  directly from the definition of the product. The idea that $T_\sigma$
  represents in some sense the ``square'' of the involution $\sigma$ has
  appeared in the literature in various forms, for instance in the definition
  of the signature of an involution (see \cite{AQM2}). Our construction gives
  some solid ground to this idea.
\end{ex}

\begin{coro}\label{cor_prod_trace}
  Let $(A,\sigma)$ be an algebra with involution over $K$, $V$ be a right
  $A$-module, and $h,h'$ be two $\eps$-hermitian forms on $V$. If we set
  $B=\End_A(V)$ and $\tau=\sigma_h$, then there is a unique $u\in B^\times$
  such that for all $x,y\in V$, $h'(x,y)=h(ux,y)$. Furthermore, $\tau(u)=u$, and
  \[ h'\cdot h = T_{\tau,u} \]
  as a product in $\tld{GW}(A,\sigma)$. In particular, $h^2=T_\tau$.
\end{coro}

\begin{proof}
  For a fixed $x$, $h'(x,-)$ is an $A$-linear map $V\to A$, so since $h$ is
  regular there exists a unique $x'$ such that $h'(x,-)=h(x',-)$. It is
  easy to see that $x\mapsto x'$ is $A$-linear, which shows existence and
  uniqueness of $u$ (which is invertible since $h'$ is also regular). We easily
  see that $\tau(u)=u$ using that $h'$ is $\eps$-hermitian.
  
  Let $f:(B,\tau)\to (A,\sigma)$ be the morphism in $\CBrh$ corresponding
  to $(V,h)$. Then $f_*(\fdiag{1}_\tau)=(V,h)$, and
  $f_*(\fdiag{b}_\tau)=h'$ (see Example \ref{ex_comp_diag}).
  Thus since $f_*$ is a ring morphism we find
  $h'\cdot h = \fdiag{u}_\tau\cdot \fdiag{1}_\tau = T_{\tau,u}$.
\end{proof}

This means that we can reinterpret twisted involution forms in the sense of
\cite[Â§11]{BOI} as being exactly the products of $\eps$-hermitian forms defined on 
the same module. The previous computations show that understanding the product in 
$\tld{GW}(A,\sigma)$ and $\tld{W}(A,\sigma)$ amounts to understanding those twisted 
involution trace forms (usually for involutions different from $\sigma$).

\subsection{Quaternion algebras}\label{sec_quater}

We can in particular perform these computations for quaternion algebras,
imitating the proof of \cite[11.6]{BOI}. Recall that if $Q$ is a quaternion
algebra, then its reduced norm map is a quadratic form on $Q$, denoted
$n_Q\in GW(K)$, and it is the unique 2-Pfister form whose Clifford invariant
$e_2(n_Q)\in H^2(K,\mu_2)$ is the Brauer class of $Q$.

For any pure quaternions $z_1,z_2\in Q$, the Brauer class $[Q]$ and
the symbol $(z_1^2,z_2^2)\in H^2(K,\mu_2)$ have a common slot (for instance $z_1^2$),
so $[Q]+(z_1^2,z_2^2)$ is a symbol. We write $\phi_{z_1,z_2}$ for the
unique 2-Pfister form whose Clifford invariant is this symbol.
In particular, if $z_1$ and $z_2$ anti-commute, $\phi_{z_1,z_2}$ is
hyperbolic.

\begin{prop}\label{prop_prod_quater}
  Let $(Q,\gamma)$ be a quaternion algebra over $K$ endowed with its
  canonical symplectic involution. Then for any $a,b\in K^*$ we have
  in $\tld{GW}(Q,\gamma)$:
  \[ \fdiag{a}_\gamma\cdot \fdiag{b}_\gamma = \fdiag{2ab}n_Q \in GW(K). \]
  Furthermore, for any pure quaternions $z_1,z_2\in Q^\times$, 
  $\fdiag{z_1}_\gamma\cdot \fdiag{z_2}_\gamma$ is similar to $\phi_{z_1,z_2}$.
  When $z_1$ and $z_2$ anti-commute, this means $\fdiag{z_1}_\gamma\cdot 
  \fdiag{z_2}_\gamma$ is hyperbolic. When they do not anti-commute, we get:
  \[ \fdiag{z_1}_\gamma\cdot \fdiag{z_2}_\gamma = 
  \fdiag{-\Trd_Q(z_1z_2)}\phi_{z_1,z_2} \in GW(K). \]            
\end{prop}

\begin{proof}
  From Proposition \ref{prop_prod_diag} we see that
  \[ \fdiag{a}_\gamma\cdot \fdiag{b}_\gamma = \fdiag{ab}T_\gamma, \]
  and it is easy to see that $T_\gamma = \fdiag{2}n_Q$ (for instance by
  taking a standard quaternionic basis of $Q$). For the second formula,
  we can check that if $z_1$ and $z_2$ anti-commute then $1$ and $z_1$ span
  a Lagrangian in $(Q,T_{\gamma,z_1,z_2})$, and otherwise we can check that
  if $z\in Q^\times$ anti-commutes with $z_1$, then the basis $(1,z_1,zz_2,z_1zz_2)$
  is orthogonal for $T_{\gamma,z_1,z_2}$, giving the diagonalization
  \[ \fdiag{-\Trd_Q(z_1z_2)}\fdiag{1,-z_1^2,-z_2^2z^2,z_1^2z_2^2z^2} \]
  which shows that $T_{\gamma,z_1,z_2}\simeq \fdiag{-\Trd_Q(z_1z_2)}\pfis{z_1,z_2z}$,
  and it is easy to see that by definition $\phi_{z_1,z_2}=\pfis{z_1,z_2z}$.
\end{proof}


Using this result, we can also compute products in $(Q,\sigma)$ for 
orthogonal involutions $\sigma$. Indeed, we know that $\sigma=\gamma_u$
for some pure quaternion $u\in Q^\times$; then we have a Morita equivalence
$f: (Q,\sigma)\mor{\fdiag{u}_\gamma} (Q,\gamma)$. From Example \ref{ex_comp_diag}
we deduce that $f_*(\fdiag{z}_\sigma) = \fdiag{uz}_\gamma$. So we can simply
compute $\fdiag{z}_\sigma\cdot \fdiag{z'}_\sigma = 
\fdiag{uz}_\gamma\cdot \fdiag{uz'}_\gamma$ using Proposition \ref{prop_prod_quater}
(it does not give especially enlightening formulas).

\begin{rem}
  In \cite{Lew}, Lewis gives a description in terms of generators and relations
  of $\tld{W}(Q,\gamma)$ in a few simple cases, namely when the base field is
  real closed or is a $p$-adic field (recall though that there is a factor $\fdiag{2}$
  between his product and ours). Unfortunately, it does not really seem feasible
  to give other such complete descriptions, simply because the underlying Witt
  groups themselves (let alone the ring structure) are to our knowledge only described 
  as precisely in those exact cases.
\end{rem}

\subsection{Scalar extension and reciprocity}\label{sec_recip}

Scalar extension is a standard tool in the theory of algebras with
involution, in particular when extending the scalars to a splitting
field to reduce to the classical theory of bilinear forms over fields.

When studying the usual Witt rings of fields, it is also standard to consider
the ring morphisms $\rho: GW(K)\to GW(L)$ induced by a field extension $L/K$. One
of the basic facts (see \cite[2.5.6]{Sch}) is that we can also go in the other 
direction provided the extension is finite: if $s:L\to K$ is a non-zero $K$-linear map, 
then one can define a group morphism
\[ \foncdef{s_*}{GW(L)}{GW(K)}{(V,b)}{(V,s\circ b),} \]
usually called a Scharlau transfer map,
and the Reciprocity Theorem states that this is a $GW(K)$-module morphism.
In concrete terms, this means that if $x\in GW(K)$ and $y\in GW(L)$ then
\[ s_*(\rho(x)\cdot y) = x\cdot s_*(y). \]

Of course this is reminiscent of other such reciprocity phenomena, such as
Frobenius reciprocity for induction/restriction of group representations, or
the projection formula for cup-products in cohomology.
\\

Let us explain how such reciprocity formulas arise inside the framework we 
established in Section \ref{sec_grad} to construct graded rings. Consider commutative monoids
$M$ and $N$, and $F,G\in \LaxHom_{\otimes}^s(\cat{M}, \mcat[\Z]_N)$. Then
if $\phi: F\to G$ is a monoidal transformation, it induces an $(M\times N)$-graded 
ring morphism $\phi_*: A_F\to A_G$, where $A_F$ and $A_G$ are the graded rings 
corresponding to $F$ and $G$. In particular, $A_G$ is a (graded) module over $A_F$.
Now if $\psi: G\to F$ is a natural transformation (not monoidal) in the other direction, 
it induces a graded group morphism $\psi_*: A_G\to A_F$. When is this an $A_F$-module
morphism? It is easy to check that the relevant condition is that the following
diagram commutes for all $x,y\in M$:
\begin{equation}\label{diag_recipr}
  \begin{tikzcd}
    F(x)\otimes G(y) \rar{\phi\otimes \Id} \dar{\Id\otimes \psi} 
    & G(x)\otimes G(y) \rar & G(xy) \dar{\psi} \\
    F(x)\otimes F(y) \arrow[rr] & & F(xy).
  \end{tikzcd}
\end{equation}

Now let us apply this to our situation. Let $L/K$ be any field extension.
We have an obvious scalar extension functor $\CBrh\to \CBrh[L]$, which fits in 
commutative diagrams of functors
\[ \begin{tikzcd}
  \mathbf{AlgInv}(K) \rar{\Theta} \dar & \CBrh \dar \\
  \mathbf{AlgInv}(L) \rar{\Theta}  & \mathbf{Br}_h(L)
\end{tikzcd} \]
and
\[ \begin{tikzcd}
  \CBrh \rar \dar & {\CBrh}[2] \dar \rar & \CBrh \dar \\
  \CBrh[L] \rar & {\CBrh[L]}[2] \rar & \CBrh[L]
\end{tikzcd} \]
using the canonical coherent $2$-torsion structures.

For any algebra with involution $(A,\sigma)$, scalar extension of hermitian
modules yields group morphisms
\begin{equation}
  \rho: GW_\eps(A,\sigma)\To GW_\eps(A_L,\sigma_L)
\end{equation}
which fit into a $2$-cell
\[ \begin{tikzcd}
  \CBrh \rar{GW_\pm} \dar & \mcat[\Z]_{\mu_2(K)} \arrow[d, equal] 
      \arrow[dl,Rightarrow, shorten <=0.6em, shorten >=0.6em, "\rho"] \\
  \CBrh[L] \rar[swap]{GW_\pm} & \mcat[\Z]_{\mu_2(L)}
\end{tikzcd} \]
In fact, as scalar extension is compatible with
tensor products, $\rho$ is even a monoidal natural transformation.
\\

Now let us assume that $L/K$ is finite, and let $s:L\to K$ be a non-zero
$K$-linear form. Then for any $K$-algebra $A$, $s$ extends naturally to a
$K$-linear map $s_A: A_L\to A$ by $\Id_A\otimes s$, and this defines an
\emph{involution trace} in the sense of \cite[4.3]{BOI}. This shows (see also
\cite[I.7.2,I.7.3.2]{Knu}) that we can define group morphisms, which are
the analogues of the Scharlau transfer,
\begin{equation}\label{eq_def_s}
  s_*: GW_\eps(A_L,\sigma_L)\To GW_\eps(A,\sigma)
\end{equation}
by sending $(V,h)$ over $(A_L,\sigma_L)$ to $(V,s_A\circ h)$.

\begin{prop}
  The morphisms in (\ref{eq_def_s}) form a natural transformation
  which fits into a $2$-cell
  \[ \begin{tikzcd}
    \CBrh \rar{GW_\pm} \dar & \mcat[\Z]_{\mu_2(K)} \arrow[d, equal]  \\
    \CBrh[L] \rar[swap]{GW_\pm} 
    \arrow[ur,Rightarrow, shorten <=0.6em, shorten >=0.6em, "s_*"] & \mcat[\Z]_{\mu_2(L)}.
  \end{tikzcd} \]
\end{prop}

\begin{proof}
  If $f: (B,\tau)\mor{(V,h)} (A,\sigma)$ is a morphism in $\CBrh$, we need
  to show that we get a commutative diagram
  \[ \begin{tikzcd}
    GW_\eps(B_L,\tau_L) \dar[swap]{(f_L)_*} \rar{s_*} & GW_\eps(B,\tau) \dar{f_*} \\
    GW_\eps(A_L,\sigma_L) \rar{s_*} & GW_\eps(A,\sigma).
  \end{tikzcd} \]
  Let $(U,g)\in GW_\eps(B_L,\tau_L)$. Then we need to give an isometry from
  $(U\otimes_B V,\alpha)$ to $(U\otimes_{B_L} V_L,\beta)$ with
  \[ \alpha(u\otimes v, u'\otimes v') = h(v, s_B(g(u,u'))\cdot v')  \]
  and
  \[ \beta(u\otimes (v\otimes \lambda), u'\otimes (v'\otimes \mu)) = 
    s_A(h_L(v\otimes \lambda, g(u,u')\cdot (v'\otimes \mu))).  \]
  We consider $\phi: U\otimes_B V\to U\otimes_{B_L} V_L$ defined by
  $u\otimes v\mapsto u\otimes (v\otimes 1)$; to show that it is an isometry from $\alpha$
  to $\beta$, we need to show the equality:
  \[ s_A(h_L(v\otimes 1, g(u,u')\cdot (v'\otimes 1))) = 
    h(v, s_B(g(u,u'))\cdot v'). \]
  We can write $g(u,u')=\sum b_i\otimes \lambda_i$, and since both expressions are
  additive in $g(u,u')$, we may actually assume that $g(u,u')=b\otimes \lambda$.
  But then the left-hand side is 
  \[ s_A(h_L(v\otimes 1,bv'\otimes \lambda)) = 
  s_A(h(v,bv')\otimes \lambda) = h(v,bv')s(\lambda)\]
  and the right-hand side is
  \[ h(v, s_B(b\otimes \lambda)\cdot v') = h(v,s(\lambda)bv')  \]
  which concludes.
\end{proof}

Thus we have at our disposal two natural transformations $\rho$ and $s_*$ going in 
opposite directions between the same functors, and one of them, namely $\rho$ is 
monoidal. The following proof is essentially an adaptation of \cite[2.5.6]{Sch}.

\begin{prop}\label{prop_recipr}
  The natural transformations $\rho$ and $s_*$ satisfy the relationship
  given by the diagram (\ref{diag_recipr}), with $F(A,\sigma) = GW_\pm(A,\sigma)$
  and $G(A,\sigma) = GW_\pm(A_L,\sigma_L)$.
\end{prop}

\begin{proof}
  We need to show that if $(V,h)\in GW_\eps(A,\sigma)$ and $(U,g)\in GW_{\eps'}(B_L,\tau_L)$
  then 
  \[ s_{A\otimes_K B}\circ (h_L\otimes g) = h\otimes (s_B\circ g) \in 
    GW_{\eps\eps'}(A\otimes_K B,\sigma\otimes \tau). \]
  It is easy to see that the map
  \[ \anonfoncdef{(V\otimes_K L)\otimes_L U}{V \otimes_K U}
    {(x\otimes \lambda)\otimes y}{x\otimes (\lambda y)} \]
  defines an isometry.
\end{proof}

We finally establish the desired reciprocity result:

\begin{thm}[Frobenius Reciprocity]\label{thm_recipr}
  Let $L/K$ be a finite field extension, and $s:L\to K$ a non-zero $K$-linear form.
  Then for any algebra with involution $(A,\sigma)$ over $(K,\Id)$, the scalar
  extension map $\rho: \tld{GW}(A,\sigma)\to \tld{GW}(A_L,\sigma_L)$ is a $\Gamma$-graded
  ring morphism, and the transfer $s_*: \tld{GW}(A_L,\sigma_L)\to \tld{GW}(A,\sigma)$
  is a morphism of graded $\tld{GW}(A,\sigma)$-modules.

  In practice, this means that for all $x\in \tld{GW}(A,\sigma)$ and 
  $y\in \tld{GW}(A_L,\sigma_L)$, we have:
  \[ s_*(\rho(x)\cdot y) = x\cdot s_*(y). \]
\end{thm}

\begin{proof}
  The canonical coherent $2$-torsion structures associate to each $(A,\sigma)$ 
  symmetric monoidal functors 
  \[ \begin{tikzcd}
    \cat{\Zd} \rar \drar & \CBrh \dar \\
    & \CBrh[L].
  \end{tikzcd} \]
  If we use the notation in the statement of Proposition \ref{prop_recipr},
  then by composition we get symmetric monoidal functors $F',G': \cat{\Zd}\to 
  \mcat[\Z]_{\mu_2(L)}$, and $\rho$ and $s_*$ give natural transformation
  $\phi: F'\to G'$ and $\psi: G'\to F'$ which satisfy the relation in 
  (\ref{diag_recipr}). By construction, the graded rings associated to $F'$ and $G'$ 
  are precisely $\tld{GW}(A,\sigma)$ and $\tld{GW}(A_L,\sigma_L)$, and the natural 
  transformations $\phi$ and $\psi$ correspond to the maps given by $\rho$ and $s_*$.

  Then as discussed above, the relation (\ref{diag_recipr}) exactly gives the
  fact that $\rho$ is a ring morphism and $s_*$ is a $\tld{GW}(A,\sigma)$-module
  morphism.
\end{proof}

Since both $\rho$ and $s_*$ send hyperbolic forms to hyperbolic forms, they also
induce maps $\rho: \tld{W}(A,\sigma)\to \tld{W}(A_L,\sigma_L)$ and 
$s_*: \tld{W}(A_L,\sigma_L)\to \tld{W}(A,\sigma)$, which also satisfy
the reciprocity formula, since it holds in $\tld{GW}$.

\begin{rem}\label{rem_trace_ideal}
  In particular, the image of the transfer map $s_*$ is an ideal in $\tld{GW}(A,\sigma)$
  (resp. $\tld{W}(A,\sigma)$), which as in the classical case we call the \emph{trace ideal}
  relative to $L/K$, and it can be shown that it does not depend on the choice of $s$.
\end{rem}

\bibliographystyle{plain}
\bibliography{witt_mixte}

\end{document}